\documentclass[11pt]{amsart}

\usepackage{amssymb,amsmath,amsthm,amsfonts,enumerate}
\newtheorem{thm}{Theorem}[section]

\newtheorem{lem}[thm]{Lemma}

\theoremstyle{definition}

\theoremstyle{remark}

\numberwithin{equation}{section}

\begin{document}
\title[Solitary-Wave Solutions of a Schr\"{o}dinger-KdV system]{Existence and Stability of
a \\ Two-Parameter Family of  Solitary Waves \\
for an NLS-KdV System} 
\author[John Albert]{John Albert}
\address{Department of Mathematics, University of Oklahoma, 601 Elm Ave,
Norman, OK 73019 USA}
\email{jalbert@ou.edu}
\author[Santosh Bhattarai]{Santosh Bhattarai}
\address{Department of Mathematics, University of Oklahoma,
601 Elm Ave, Norman, OK 73019 USA}
\email{sbhattarai@ou.edu}


\keywords {Schr\"{o}dinger-KdV system, concentration-compactness principle, solitary waves, stability, variational problems}
\thanks{\textit{MSC2010}. Primary 35B35, 35Q53, 35Q55 ; Secondary 35A15, 76B25}

\begin{abstract}
We prove existence and stability results for a two-parameter
family of solitary-wave solutions to a system in which an equation
of nonlinear Schr\"odinger type is coupled to an equation of
Korteweg-de Vries type.  Such systems model interactions between
short and long dispersive waves.   The results extend earlier
results of Angulo, Albert and Angulo, and Chen.  Our proof involves the
characterization of solitary-wave solutions as minimizers of an
energy functional subject to two constraints. To establish the
precompactness of minimizing sequences via concentrated
compactness, we establish the sub-additivity of the problem with
respect to both constraint variables jointly.
\end{abstract}

 \maketitle


\section{Introduction}

 Both the nonlinear Schr\"odinger equation
\begin{equation}\label{NLS}
iu_t + u_{xx} + |u|^q u = 0
\end{equation}
for a complex-valued function $u$ of $x \in \mathbb{R}$ and time $t$,
and the generalized Korteweg-de Vries equation
\begin{equation}\label{GKdV}
v_t +v_{xxx} + v^pv_x = 0,
\end{equation}
for a real-valued function $v$ of $x$ and $t$, are universal
models for nonlinear waves in dispersive media.  Equation
\eqref{GKdV} arises generically as a model for waves whose motion,
to first order, is governed by the linear wave equation $v_t + v_x
= 0$, but which on account of their long wavelength and small but
finite amplitude are influenced by weak nonlinear and dispersive
effects. Equation \eqref{GKdV}, on the other hand, describes the
amplitude and phase modulations of long-wavelength,
small-amplitude perturbations of a monochromatic short wave in a
dispersive medium. Discussions of the canonical nature of these
equations may be found, for example, in chapters 13 and 17 of
\cite{[W]}, chapter 2 of \cite{[N]}, or chapter 10 of \cite{[M]}.

In this paper we will consider a system describing the interaction of a nonlinear Schr\"odinger-type wave with a Korteweg-de Vries type wave:
\begin{equation}\label{NLSKDV}
\begin{aligned}
iu_{t}+u_{xx}+\tau_1 |u|^q u&=-\alpha uv \\
v_{t}+v_{xxx}+\tau_2 v^p v_{x}&=-\dfrac{\alpha}{2}(|u|^2)_{x},
\end{aligned}
\end{equation}
where $\tau_1$, $\tau_2$, and $\alpha$ are real constants.
The form of the coupling terms in system \eqref{NLSKDV} is also
universal:  the system arises as a model for interactions between
long waves and long-wavelength envelopes of short waves in a
variety of physical settings. For example, it appears in
\cite{[FO]} and \cite{[KSK]} as a model for the interaction
between long gravity waves and capillary waves on the surface of
shallow water. A system of similar form, but with the term
$v_{xxx}$ in the second equation replaced by $-v_{xxx}$, appears
in \cite{[AV]} (see also \cite{[NHMI]}) as a model for the
interaction of Langmuir waves and ion-acoustic waves in a plasma.
The status of \eqref{NLSKDV} as a generic model may be related to
the fact that it has a Hamiltonian structure in which the
Hamiltonian (the functional $E(u,v)$ defined below) has the
coupling term $\alpha v|u|^2$.  If one requires the coupling term
to be a power series in $|u|^2$ and $v$, this is the simplest
possible coupling one could expect.

We consider here the initial-value problem for \eqref{NLSKDV} on
the line, for $(u,v)$ in the space $Y=H^1_\mathbb{C}(\mathbb R)
\times H^1(\mathbb R)$. (Here $H^1(\mathbb R)$ and
$H^1_\mathbb{C}(\mathbb R)$ are $L^2$-based Sobolev spaces of
real- and complex-valued functions on the line, respectively. For
more details on our notation, see below.)  In the case when $p=1$
and $q=2$, for arbitrary values of $\tau_1$, $\tau_2$, and
$\alpha$, this problem has been shown to be well-posed locally in
time by Bekiranov et al. in \cite{[BOP]}, and global
well-posedness was proved by Corcho and Linares \cite{[CL]}. Dias,
Figueira and Oliveira \cite{[DFO]} extended the global
well-posedness result to the case when $p=1$ and $1 < q < 4$, and
their proof will work for $q=1$ as well\footnote[1]{Jo\~ao Paulo
Dias, personal communication.}. These results depend on the fact
that the following functionals are conserved under the flow of
\eqref{NLSKDV}:
\begin{equation}
\label{defE} E(u,v)=\int_{-\infty }^{\infty
}{\left(|u_{x}|^{2}+v_{x}^{2}- \beta_1 |u|^{q+2}-\beta_2
v^{p+2}-\alpha |u|^{2}v\right)\ dx},
\end{equation}
where $\beta_1 = 2\tau_1/(q+2)$ and $\beta_2 = 2\tau_2/((p+1)(p+2))$,
\begin{equation}
\label{defG} G(u,v)=\int_{-\infty }^{\infty }v^{2}\
dx+\text{Im}\int_{-\infty }^{\infty }u \overline u_{x}\ dx,
\end{equation}
where $\overline u_x$ is the complex conjugate of $u_x$ and $\text{Im}(z)$ denotes the imaginary part of $z$, and
\begin{equation}
\label{defH}
H(u)=\int_{-\infty }^{\infty }\left\vert u\right\vert ^{2}dx.
\end{equation}
In other words, if $(u,v)$ is a solution of \eqref{NLSKDV} in $Y$, then
$E(u,v)$, $G(u,v)$, and $H(u)$ are independent of time.

The methods used in this paper require the assumption that $\tau_1$, $\tau_2$
and $\alpha$ are positive, or at least non-negative. Also, in order for the term $v^p$ in \eqref{NLSKDV} or
\eqref{defE} to be defined when $v<0$, we will assume in what
follows that $p$ is a positive rational number with odd
 denominator.  Much of what is proved below should be readily
 extendable to versions of \eqref{NLSKDV} with other nonlinearities, such as
 when $v^pv_x$ is replaced in \eqref{NLSKDV} by $(|v|^p)_x$, in
 which case the analogue of Theorem \ref{existence} will hold for
 all real values of $p$ such that $1 \le p < 4$.

The purpose of this paper is to prove existence and stability results
for (coupled) solitary traveling-wave solutions of \eqref{NLSKDV}.  Such a solution is of the form
\begin{equation}\label{SO}
(u(x,t),v(x,t))=\left( e^{i\omega t}e^{ic(x-ct)/2}\phi (x-ct),\psi
(x-ct)\right) ,
\end{equation}
where $c>0,$ $\omega \in \mathbb{R}$, and $\phi:\mathbb{R} \to \mathbb{C}$ and $\psi :\mathbb{R}\to \mathbb{R}$
are functions that vanish at infinity, in the sense that $\phi \in H^1_\mathbb{C}$ and $\psi \in H^{1}$. Inserting the ansatz
\eqref{SO} into \eqref{NLSKDV}, we see that $(u,v)$ is a solution of \eqref{NLSKDV} if and
only if $\phi$ and $\psi$ satisfy the system of ordinary
differential equations
\begin{equation}\label{ODE}
\begin{aligned}
-\phi''+\sigma \phi & =\tau_1|\phi|^q\phi+\alpha \phi\psi \\
-\psi''+c\psi &=\frac{\tau_2}{p+1}\psi
^{p+1}+\dfrac{\alpha}{2}|\phi|^{2} ,
\end{aligned}
\end{equation}
where $\sigma =\omega -c^{2}/4$, and primes denote derivatives of a function of a single variable.

One question we address below is whether nontrivial solutions of
\eqref{ODE} exist.  Our existence result is obtained by studying
the variational problem of finding, for given positive values of
$s$ and $t$, minimizers of $E(u,v)$ subject to the constraints
that $\int_{-\infty }^{\infty }{|u|^2\ dx}=s$ and $\int_{-\infty
}^{\infty }{v^2\ dx}=t$.  The connection to solitary waves is due
to the fact that equations \eqref{ODE} are the Euler-Lagrange
equations for this variational problem, with $\sigma$ and $c$
playing the role of Lagrange multipliers.  In Section
\ref{sec:existence}, we use the method of concentration
compactness to prove the relative compactness of minimizing
sequences for the variational problem, and hence the existence of
minimizers.  This requires proving the strict subadditivity (see
Lemma \ref{subadd} below) of the function $I(s,t)$ defined for
$s>0$ and $t>0$ by
\begin{equation}
\label{Ist}
I(s,t)=\inf \left\{ E(f,g): (f,g) \in Y, \int_{-\infty }^{\infty }|f|^2\ dx=s, \ \text{and  $\int_{-\infty }^{\infty }g^2\ dx=t$}\right\}.
\end{equation}

For equations \eqref{NLS} or \eqref{GKdV}, the variational
problems which characterize solitary waves depend on a single
constraint parameter, and proofs of strict subadditivity are
accomplished by simple arguments, dating back to Lions' original
paper \cite{[L]}, which take advantage of homogeneities present in
the equation.  To prove strict subadditivity for the two-parameter
problem defined in \eqref{Ist}, however, seems to be more
difficult.  In \cite{[AA]}, which treats the case where $p=1$ and
$\tau_1 = 0$, it was noted that strict subadditivity, as defined
below in Lemma \ref{subadd}, holds for $\alpha = 1/6$
(corresponding to setting the parameter $q$ in \cite{[AA]} equal
to 2), and it was shown that strict subadditivity continues to
hold for $\alpha$ in some neighborhood of 1/6.  Here we are able
to extend this result to all positive values of $\alpha$, all
non-negative values of $\tau_1$, all positive valued of $\tau_2$, all $p \in [1,4)$,
and all $q \in [1,4)$.
To do so, we rely on an argument due to
Byeon \cite{[By]} and Garrisi \cite{[G]}, which exploits the fact
that the $H^1$ norms of certain functions are strictly decreased
when the mass of the function is rearranged by symmetrization.

Before stating our existence result, let us define a minimizing
sequence for $I(s,t)$ to be a sequence $(f_n,g_n)$ in $Y$ such
that
\begin{equation*}
\lim_{n\to \infty }\int_{-\infty}^\infty |f_{n}|^2=s,\ \lim_{n\to
\infty }\int_{-\infty}^\infty g_n^2=t,\text{ and }\lim_{n\to
\infty }E(f_n,g_n)=I(s,t).
\end{equation*}

Our existence result is the following.

\begin{thm}
\label{existence} Suppose $\alpha >0$, $\tau_1 \ge 0$, $\tau_2 >
0$, $1 \le q < 4$, and $1\le p<4$, where $p$ is a rational number
with odd denominator.   For $s>0$ and $t>0$,
define
\begin{equation}
\mathcal{S}_{s,t}=\left\{(\phi,\psi) \in Y: E(\phi,\psi)=I(s,t), \int_{-\infty }^{\infty }|\phi|^2\ dx=s, \text{and  $\int_{-\infty }^{\infty }\psi^2\ dx=t$}\right\}.
\end{equation}
Then the following statements are true for all $s>0$ and $t>0$.

(i)
The infimum $I(s,t)$ defined in \eqref{Ist} is finite.

(ii)
Every minimizing sequence $\{(f_{n},g_{n})\}$ for
$I(s,t)$ is relatively compact in $Y$ up to translations.  That is,
there exists a subsequence $\{(f_{n_{k}},g_{n_{k}})\}$ and a
sequence of real numbers $\{y_{k}\}$ such that $\{(f_{n_{k}}(\cdot
+y_{k}),g_{n_{k}}(\cdot +y_{k})\}$ converges strongly in $Y$ to
some $(\phi,\psi)$ in $\mathcal{S}_{s,t}$.  In particular, the set
$\mathcal{S}_{s,t}$ is non-empty.

(iii)
Each function
$(\phi,\psi) \in \mathcal{S}_{s,t}$ is a solution of \eqref{ODE}
for some $\sigma$ and $c$, and therefore when substituted into \eqref{SO} yields a solitary-wave
solution of \eqref{NLSKDV}.

(iv)
For every $(\phi,\psi)$ in $\mathcal{S}_{s,t}$, we have that $\psi(x) > 0$
for all $x \in \mathbb{R}$, and there exist a number
$\theta \in \mathbb{R}$ and a function $\tilde \phi$ such that $\tilde \phi(x)>0$
for all $x \in \mathbb R$, and $\phi(x)=e^{i\theta}\tilde\phi(x)$.  Also, the functions
$\psi$ and $\phi$ are infinitely differentiable on $\mathbb R$.
 \end{thm}

Notice that it is obvious from the definition of the sets
$\mathcal{S}_{s,t}$ that they form a true two-parameter family, in that
$\mathcal{S}_{s_1,t_1}$ and $\mathcal{S}_{s_2,t_2}$ are disjoint
if $(s_1,t_1) \ne (s_2,t_2)$.
Previously, Dias et al. \cite{[DFO]} had proved that for $p \in
\{1, 2, 3\}$ (with $\alpha
> 3$ if $p=1$), \eqref{NLSKDV} has an infinite family of positive
bound states which decay exponentially at infinity. Compared to
the result of \cite{[DFO]}, ours has the advantages that we do not
require $\alpha > 3$ when $p=1$, and also that we obtain a true
two-parameter family of solitary waves.  In \cite{[DFO]}, nonempty
sets $\mathcal{T}_{\delta,\mu}$ of solitary waves are obtained by
minimizing $E$ subject to $\int |u|^2 + \delta v^2 = \mu$,  but it
is not clear whether $\mathcal{T}_{\delta_1,\mu_1}$ is necessarily
disjoint from $\mathcal{T}_{\delta_2,\mu_2}$ if $(\delta_1,\mu_1)
\ne (\delta_2, \mu_2)$.

A separate question is that of stability of the solutions of \eqref{ODE} as solutions of the initial-value problem for
\eqref{NLSKDV}.  For $s > 0$ and $t \in \mathbb{R}$, define
\begin{equation}
\label{Wst}
W(s,t)=\inf \{E(h,g):(h,g)\in Y,\ H(h)=s\text{ and }G(h,g)=t\}.
\end{equation}
The variational problem associated to $W(s,t)$ is suitable for
studying stability because not only the functional $E$ being
minimized, but also the constraint functionals $G$ and $H$ are
conserved for \eqref{NLSKDV}.  If minimizers $(\Phi,\psi)$ for
$W(s,t)$ exist, they satisfy the Euler-Lagrange equations
\begin{equation}
\begin{aligned}
-\Phi'' +\omega \Phi + c i \Phi' &=  \tau_1|\Phi|^q \Phi +  \alpha \Phi\psi \\
-\psi'' + c \psi &= \frac{\tau_2 \psi^{p+1}}{p+1} +
\frac{\alpha}{2}|\Phi|^2
\end{aligned}
\label{ELforWst}
\end{equation}
where the real numbers $c$ and $\omega$ are the Lagrange
multipliers.  These equations are satisfied by $\Phi$ and $\psi$
if and only if the functions $u$ and $v$ defined by
\begin{equation}\label{SO2}
(u(x,t),v(x,t))=\left( e^{i\omega t}\Phi(x-ct),\psi(x-ct)\right)
\end{equation}
are solutions of the NLS-KdV system \eqref{NLSKDV}.  That is,
solutions $(\Phi,\psi)$ of the variational problem for $W(s,t)$
are solitary-wave profiles, and \eqref{SO} is recovered from
\eqref{SO2} by setting $\Phi(x)=e^{icx/2}\phi(x)$.

We have the following stability result.

\begin{thm}\label{stability}Suppose $\alpha >0$, $\tau_1 \ge 0$, $\tau_2 > 0$, $1 \le q < 4$, and $p=1$.
For $s>0$ and $t \in \mathbb{R}$, define
\begin{equation}
\mathcal{F}_{s,t}=\left\{(\Phi,\psi) \in Y: E(\Phi,\psi)=W(s,t),
H(\Phi)=s, \text{and  $G(\Phi,\psi)=t$}\right\} \label{Fst}.
\end{equation}
Then the following statements are true for all $s > 0$ and $t \in \mathbb R$.

(i) The infimum $W(s,t)$ defined in \eqref{Wst} is finite.

(ii) Every minimizing sequence $\{(h_{n},g_{n})\}$ for
$W(s,t)$ is relatively compact in $Y$ up to translations.  That
is, if
\begin{equation*}
\lim_{n\to \infty }H(h_n)=s,\ \lim_{n\to \infty}
G(h_n,g_n)=t,\text{ and }\lim_{n\to \infty }E(h_n,g_n)=W(s,t),
\end{equation*}
then there is a subsequence $\{(h_{n_k},g_{n_k})\}$ and a sequence
of real numbers $\{y_k\}$ such that
$\{h_{n_k}(\cdot+y_k),g_{n_k}(\cdot+y_k)\}$ converges strongly in
$Y$ to some $(\Phi,\psi) \in \mathcal{F}_{s,t}$.  In particular, the set
$\mathcal{F}_{s,t}$ is non-empty.

(iii)
 Each $(\Phi,\psi) \in
\mathcal{F}_{s,t}$ is a solution of \eqref{ELforWst} for some
$\omega$ and $c$, and therefore when substituted into \eqref{SO2} yields a solitary-wave solution
of \eqref{NLSKDV}.

(iv) For every $(\Phi,\psi) \in
\mathcal{F}_{s,t}$, let $a=\|\psi\|^2$ and $b=(t-a)/s$.  Then there exist $\theta \in \mathbb R$
and a real-valued function $\tilde \phi$ such that $(\tilde \phi,\psi) \in \mathcal S_{s,a}$ and
\begin{equation}
\Phi(x)=e^{i(-bx + \theta)}\tilde\phi(x)
\label{formofphi}
\end{equation}
on $\mathbb R$.  Further, if $\tau_1 = 0$, then $a>0$, $\psi(x)>0$ for all $x \in \mathbb R$, and
we can take $\tilde\phi$ to be everywhere positive on $\mathbb R$.

(v)  The set $\mathcal{F}_{s,t}$ is a
stable set of initial data for \eqref{NLSKDV}, in the following sense: for every $\epsilon >0$, there
exists $\delta>0$ such that if $(h_0 ,g_0 )\in Y$,
\begin{equation*}
\inf_{(\Phi,\psi)\in \mathcal{F}_{s,t}}\|(h_0,g_0)-(\Phi ,\psi)\|_{Y}<\delta ,
\end{equation*}
and $(u(x,t),v(x,t))$ is the
solution of \eqref{NLSKDV}
with
\begin{equation*}
(u(x,0),v(x,0)) =(h_0(x),g_0(x)),
\end{equation*}
 then for all $t \ge 0$,
\begin{equation*}
\inf_{(\Phi,\psi)\in \mathcal{F}_{s,t}}\|(u(\cdot,t),v(\cdot,t))-(\Phi,\psi)\|_{Y}<\epsilon.
\end{equation*}

Furthermore,
the sets $\mathcal{F}_{s,t}$
form a true two-parameter family, in that $\mathcal{F}_{s_1,t_1}$
and $\mathcal{F}_{s_2,t_2}$ are disjoint if $(s_1,t_1) \ne
(s_2,t_2)$.
\end{thm}

We remark that, if it is assumed that that \eqref{NLSKDV} is globally well-posed in $Y$
when $1 \le p < 4/3$ (where $p$ is rational with odd denominator), then the above
stability result extends to these values of $p$ as well, with the same proof.

From the definition of the variational problem for $W(s,t)$ it is
clear that the sets $\mathcal{F}_{s,t}$ are invariant under the
transformation
\begin{equation*}
(\Phi(x),\psi(x)) \mapsto (e^{i\theta}\Phi(x-\xi),\psi(x-\xi)),
\end{equation*}
 for every pair of real numbers $\theta$ and $\xi$,
and so are at least two-dimensional in size. On the other hand,
for a given solitary-wave profile $(g,h)$ in $\mathcal{F}_{s,t}$,
the orbit $\mathcal{O}=\{(u(x,t),v(x,t)):t \in \mathbb{R}\}$ of
the corresponding solitary wave is seen from \eqref{SO2} to be
given by
\begin{equation*}
\mathcal{O}=\left\{(e^{ict}\Phi(x-ct),\psi(x-ct)):t \in
\mathbb{R}\right\},
\end{equation*}
and hence is a proper (one-dimensional) subset of
$\mathcal{F}_{s,t}$.  Therefore Theorem \ref{stability} is
somewhat weaker than an orbital stability result for the solitary
waves in $\mathcal{F}_{s,t}$.

According to part (iv) of Theorem \ref{stability}, in the case when $\tau_1 =0$, each element $(\Phi,\psi)$ of
$\mathcal F_{s,t}$ has a non-trivial second component $\psi$.  However, we have not been able to establish
this in the case when $\tau_1 >0$.  Even if $\mathcal F_{s,t}$ were to consist solely of solitary waves
of the form $(\Phi,0)$, however, it would still be of interest to know that $\mathcal F_{s,t}$ is stable in
the sense described in part (v).

Theorem \ref{stability} generalizes the stability results of
\cite{[C]}, which treated the case when $\tau_1 = 0$, $p=1$, and
$\alpha = 1/6$; and of \cite{[AA]}, which treated the case when
$\tau_1 = 0$, $p=1$, and $\alpha$ is in some neighborhood of $1/6$.
We also note the interesting paper of Angulo \cite{[An]}, which
proves stability by a different method in the case when $\tau_1 =
0$, $p=1$, $\alpha >0$, and the wavespeed $\sigma$ appearing in
\eqref{ODE} is sufficiently small.

\bigskip
The remainder of the paper is organized as follows.  In Section \ref{sec:existence}, after a number
of preparatory lemmas, including Byeon and Garrisi's rearrangement lemma, we prove assertions (i) through (iv)
of Theorem \ref{existence}.  Section \ref{sec:stability} also begins with some preparatory lemmas, and then
concludes with a proof of assertions (i) through (v) of Theorem \ref{stability}.

\bigskip

\textit{Notation}. For $1\leq p\leq \infty$, we denote by $L^{p}$
the space of all measurable functions $f$ on $\mathbb{R}$ for
which the norm $|f|_{p}$ is finite, where
\begin{equation*}
|f|_{p}=\left( \int_{-\infty }^{\infty }|f|^{p}\ dx\right)^{1/p}\text{ \ for }1\leq p<\infty
\end{equation*}
and $|f|_{\infty }$ is the essential supremum of $|f|$ on $\mathbb{R}$.
Because the $L^2$ norm appears frequently below, we use the special notation $\|f\|$ for it.  That is,
\begin{equation*}
\|f\|=\left( \int_{-\infty }^{\infty }|f|^2\ dx\right)^{1/2}.
\end{equation*}

We say that a function $f$ defined on $\mathbb{R}$ is $C^\infty$
if $f$ and all its derivatives of all orders exist everywhere on
$\mathbb{R}$.

We denote by $H_{\mathbb{C}}^{1}=H_{\mathbb{C}}^{1}(\mathbb R)$
the Sobolev space of all complex-valued functions $f$ defined on
$\mathbb R$ such that $f$ and its distributional derivative $f'$
are both in $L^2$. The norm $\|\cdot\|_1$ on $H_{\mathbb{C}}^1$ is
defined by
\begin{equation*}
\|f\|_1=\left( \int_{-\infty }^{\infty }\left( |f|^2 +
|f'|^2\right)\ dx\right) ^{1/2}.
\end{equation*}  We denote the space
of all real-valued functions $f$ in $H_{\mathbb{C}}^{1}$ by $H^1$,
and we define $Y$ to be the product space
\begin{equation*}
Y = H_{\mathbb{C}}^{1}\times H^{1},
\end{equation*}
furnished with the product norm, which we denote by $\left\Vert .\right\Vert _Y$.  That is,
\begin{equation*}
\|(h,g)\|_Y^2 = \|h\|_1^2 + \|g\|_1^2.
\end{equation*}

We occasionally use below the operation of convolution of two functions, here denoted by the symbol $\star$
and defined by
\begin{equation}
\label{defconvol}
 f \star g (x) = \int_{-\infty}^\infty
f(x-y)g(y)\ dy.
\end{equation}

In the estimates below, the letter $C$ will frequently be used to denote various constants
 whose actual values are not important for our purposes.  In particular, the value of $C$ may differ from line to line.

\section{Existence of Solitary-Wave Solutions}
\label{sec:existence}

In this section, we prove Theorem \ref{existence}.  We assume throughout the section, unless otherwise stated,
 that the assumptions of Theorem \ref{existence}
hold for the constants $\alpha$, $\tau_1$, $\tau_2$, $p$, $q$, $s$, and $t$.

\begin{lem} \label{Ibounded} Every minimizing sequence for $I(s,t)$ is bounded in $Y$.
Furthermore, one has $-\infty <I(s,t)<0$.
\end{lem}

\begin{proof}
First, observe that if $ \{(f_{n},g_{n})\}$ is a minimizing
sequence for $I(s,t)$, then $\|f_{n }\|$ and $\|g_{n }\|$ are
bounded. From the Gagliardo-Nirenberg inequality (see, for
example, Theorem 9.3 of \cite{[F]}),  we have that
\begin{equation}
\label{GLforf} |f_{n}|_{q+2}^{q+2}\leq
C\|f_{nx}\|^{q/2}\|f_{n}\|^{(q+4)/2},
\end{equation}
and since $\|f_n\|$ is constant, it follows that
\begin{equation}
\label{fLqbound}
 |f_n|_{q+2}^{q+2} \le C\|(f_n,g_n)\|_Y^{q/2}.
\end{equation}
Similarly,
\begin{equation}
\label{gLpbound}|g_{n}|_{p+2}^{p+2} \leq C \|g_{nx}\|^{p/2} \le
C\|(f_n,g_n)\|_{Y}^{p/2}.
\end{equation}
(Here, as throughout the paper, $C$ denotes various constants
which may depend on $s$ and $t$ but are independent of $f_n$ and
$g_n$.) Moreover, the same estimate \eqref{fLqbound} with $q$ replaced by 2 shows that
\begin{equation*}
 |f_n|_4^4  \le C\|f_{nx}\|\cdot \|f_n\|^3 \le C \|f_{nx}\|,
\end{equation*}
so by H\"older's inequality,
\begin{equation}
\label{mixtermbound} \int_{-\infty }^{\infty }|f_n|^2|g_n|\ dx
\leq |f_n|_4^2\cdot\|g_n\|  \leq C \|f_{nx}\|^{1/2} \le
C\|(f_n,g_n)\|_Y^{1/2}.
\end{equation}

Now
\begin{equation*}
\begin{aligned}
& \|(f_n,g_n)\|_Y^2 =\|f_n\|_1^2+\|g_n\|_1^2 \\
&\ \ =E(f_n,g_n)+\int_{-\infty}^{\infty }
\left(\beta_1|f_n|^{q+2}+\beta_2 g_n^{p+2}+\alpha |f_n|^2 g_n \right) \ dx +\|f_n\|^2+\|g_n\|^2,
\end{aligned}
\end{equation*}
and $E(f_n,g_n)$ is bounded since $ \{(f_n,g_n)\}$ is a
minimizing sequence.  Therefore from \eqref{fLqbound},
\eqref{gLpbound}, and \eqref{mixtermbound} it follows that
\begin{equation*}
\|(f_n,g_n)\|_Y^2 \le C\left(1 +
\|(f_n,g_n)\|_Y^{1/2}+\|(f_n,g_n)\|_Y^{q/2}+\|(f_n,g_n)\|_Y^{p/2}\right).
\end{equation*}
Since $q/2<2$ and $p/2 <2$, we deduce that $\|(f_n,g_n)\|_Y$ is
bounded.

Once we have shown that $ \{(f_n,g_n)\}$ is bounded in $Y$, a
finite lower bound on $E(f_n,g_n)$ also follows immediately from
\eqref{fLqbound}, \eqref{gLpbound}, and \eqref{mixtermbound}.  So
$I(s,t)>-\infty$.

Finally, to see that $I(s,t)<0$, choose $(f,g)\in Y$ such that
$\|f\|^{2}=s$, $\|g\|^{2}=t$, and $f(x)>0$ and $g(x)>0$ for all $x
\in \mathbb{R}$. For each $\theta >0,$ the functions $f_{\theta
}(x)=\theta ^{1/2}f(\theta x)$ and $g_{\theta }(x)=\theta
^{1/2}g(\theta x)$ satisfy $\|f_{\theta}\|^{2}=s$, $\|g_{\theta
}\|^{2}=t$, and
\begin{equation*}
\begin{aligned}
E(f_{\theta },g_{\theta }) &=\int_{-\infty }^{\infty
}\left(|f_{\theta x}|^{2}+g_{\theta x}^2 -\beta_1 |f_{\theta
}|^{q+2}-\beta_2 g_{\theta }^{p+2}
-\alpha|f_{\theta }|^{2}g_{\theta }\right) \ dx \\
&\leq\theta ^{2}\int_{-\infty }^{\infty
}\left(|f_x|^{2}+g_x^{2}\right)\ dx -\theta ^{1/2}\int_{-\infty
}^{\infty }\alpha |f|^{2}g \ dx.
\end{aligned}
\end{equation*}
Hence, by taking $\theta $ sufficiently small, we get $E(f_{\theta
},g_{\theta })<0,$ proving that $I(s,t)<0$.
\end{proof}

\begin{lem} \label{bdgnxbelow}
Suppose $(f_n,g_n)$ is a minimizing sequence for $I(s,t)$, where
$t>0$ and $s \ge 0$.  (Note that we do not require $s >0$ here.)
Then there exists $\delta>0$ such that $\|g_{nx}\|\geq \delta
$ for all sufficiently large $n$.
\end{lem}

\begin{proof}
If the conclusion is not true, then by passing to a subsequence we may assume there exists
a minimizing sequence for which $\displaystyle \lim_{n \to \infty} \|g_{nx}\| = 0$.
From \eqref{gLpbound}  it then follows that
\begin{equation*}
\lim_{n \to \infty} \int_{-\infty}^\infty g_n^{p+2}\ dx = 0.
\end{equation*}
Moreover, because of the elementary estimate
\begin{equation*}
|g_n|_\infty \le C \|g_n\|^{1/2}\|g_{nx}\|^{1/2},
\end{equation*}
we can write, in place of \eqref{mixtermbound},
\begin{equation}
\int_{-\infty }^{\infty }|f_{n}|^{2}|g_n|\ dx
\leq C \|f_n\|^2 \|g_n\|^{1/2}\|g_{nx}\|^{1/2} \le C\|g_{nx}\|^{1/2},
\end{equation}
from which it follows that
\begin{equation*}
\lim_{n\to \infty }\int_{-\infty }^{\infty }|f_n|^2g_n\ dx=0.
\end{equation*}
Hence
\begin{equation}
\label{Ibbelow}
\begin{aligned}
I(s,t) &=\lim_{n\to \infty }E(f_n,g_n) \\
&=\lim_{n\to \infty }\int_{-\infty }^{\infty }\left( |f_{nx}|^{2}-\beta_1
|f_n|^{q+2}\right) \ dx.
\end{aligned}
\end{equation}

Now let $\psi$ be any non-negative function such that $\|\psi\|^2=t$.  For every $\theta >0$,
the function $\psi_\theta(x)=\theta^{1/2}\psi(\theta x)$ satisfies $\|\psi_\theta\|^2=t$,
so that $I(s,t) \le E(f_n,\psi_\theta)$ for all $n$.  On the other hand, if we define
\begin{equation}
\eta = \theta^2\int_{-\infty }^{\infty }\psi_x^2\ dx-
\beta_2\theta^{p/2}\int_{-\infty }^{\infty }\psi^{p+2}\ dx,
\end{equation}
then since $p/2 < 1$, by fixing $\theta >0$ sufficiently small we can arrange that
\begin{equation}
\eta < 0.
\label{elzero}
\end{equation}

Then for all $n\in\mathbb{N}$,
\begin{equation*}
\begin{aligned}
I(s,t) &\leq E(f_n,\psi_\theta) \\
&=\int_{-\infty}^{\infty}\left( |f_{nx}|^2
-\beta_1|f_n|^{q+2}
-\theta^{1/2}\alpha |f_n|^2\psi \right) \ dx + \eta \\
& \le \int_{-\infty}^{\infty } \left(|f_{nx}|^2-\beta_1
|f_n|^{q+2}\right) \ dx + \eta.
\end{aligned}
\end{equation*}
Therefore
\begin{equation*}
\begin{aligned}
I(s,t) &\leq \lim_{n \to \infty}\int_{-\infty}^{\infty
}\left(|f_{nx}|^{2}-\beta_1 |f_n|^{q+2}\right)\ dx
+ \eta,
\end{aligned}
\end{equation*}
which contradicts \eqref{Ibbelow} and \eqref{elzero}.
\end{proof}

\begin{lem}
Suppose $g(x)$ is an integrable function on $\mathbb R$ such that
\begin{equation}
\int_{-\infty}^\infty g(x)\ dx >0.
\end{equation}
 Then for every $s >0$ there exists $f \in H^1$ such that $\|f\|^2=s$ and
\begin{equation*}
\int_{-\infty}^{\infty}\left(f_x^2 -\alpha f^2 g\right)\ dx < 0.
\end{equation*}
\label{minfforg}
\end{lem}

\begin{proof}
Let $\psi$ be an arbitrary smooth, non-negative function with compact support such that $\psi(0)=1$ and $\|\psi\|^2=s$, and for $\theta > 0$
define $\psi_\theta(x)=\theta^{1/2}\psi(\theta x)$.  Then $\|\psi_\theta\|^2 = s$, and
\begin{equation}
\label{testf}
\int_{-\infty}^\infty \left(\psi_{\theta x}^2-\psi_\theta^2g\right)\ dx = \theta^2 \int_{-\infty}^\infty \psi_x^2\ dx - \theta \int_{-\infty}^\infty\psi(\theta x)^2g(x)\ dx.
\end{equation}
But, by the Dominated Convergence Theorem,
\begin{equation*}
\lim_{\theta \to 0} \int_{-\infty}^\infty \psi(\theta x)^2g(x)\ dx = B,
\end{equation*}
where $\displaystyle B=\int_{-\infty}^\infty g(x)\ dx > 0$.  Therefore from \eqref{testf} it follows that
\begin{equation}
\int_{-\infty}^\infty \left(\psi_{\theta x}^2-\psi_\theta^2g\right)\ dx \le \theta^2 \int_{-\infty}^\infty \psi_x^2\ dx - \theta B/2
\end{equation}
for all $\theta$ in some neighborhood of 0.  Since the quantity on the right-hand side can be made negative by taking $\theta$
sufficiently small, the desired $f$ can be found by taking $f=\psi_\theta$ for a sufficiently small value of $\theta$.
\end{proof}

\begin{lem}
\label{minforJ}
 Define $J:H^1 \to \mathbb R$ by
\begin{equation}
\label{defJ}
J(g)=\int_{-\infty}^\infty\left(g_x^2 - \beta_2 g^{p+2}\right)\ dx.
\end{equation}
Let $t>0$, and let $\{g_n\}$ be any sequence of functions in $H^1$ such that
\begin{equation*}
\lim_{n \to \infty} \|g_n\|^2 = t,
\end{equation*}
and
\begin{equation*}
\lim_{n \to \infty} J(g_n) = \inf \ \left\{J(g): g \in H^1\ {\rm
and }\ \|g\|^2 = t\right\}.
\end{equation*}
Then there exists a subsequence $\{g_{n_k}\}$ and a sequence of
real numbers $y_k$ such that $g_{n_k}(x+y_k)$ converges strongly in
$H^1$ norm to $g_0(x)$, where
\begin{equation}
\label{g0}
 g_0(x)=\left(\frac{\lambda}{\beta_2}\right)^{1/p}{\rm
sech}^{2/p}\left(\frac{\sqrt{\lambda} px}{2}\right),
\end{equation}
and $\lambda > 0$ is chosen so that $\|g_0\|^2 = t$. In particular,
\begin{equation}
J(g_0)=\inf \ \left\{J(g): g \in H^1\ {\rm and }\ \|g\|^2 =
t\right\}. \label{g0ismin}
\end{equation}
\end{lem}

\begin{proof}  The proof that some subsequence of $g_n$ must converge, after suitable translations,
strongly in $H^1$ norm is by now a standard exercise in the use of the method of
concentration compactness.  A proof in the case $p=1$ appears, for example, in Theorem 2.9 of \cite{[A]}, or
Theorem 3.13 of \cite{[AA]}. A similar
proof, with obvious alterations, works for all $p \in [1,4)$ because for such $p$ the
Gagliardo-Nirenberg inequality \eqref{gLpbound} permits one to
obtain a uniform bound on $\|g_n\|_1$.

Denote the translated subsequence of $\{g_n\}$ which converges strongly by $\{g_{n_k}(x+\tilde y_k)\}$,
and let $\psi \in H^1$ be its limit.
Then $\psi$ must satisfy
\begin{equation}
\label{psiismin}
J(\psi)=\inf \ \left\{J(g): g \in H^1\ {\rm and }\ \|g\|^2 =
t\right\},
\end{equation}
and must also be a solution of the Euler-Lagrange equation
\begin{equation}
-2\psi'' -(p+2)\beta_2 \psi^{p+1} = -2\lambda \psi \label{elforg}
\end{equation}
for some real number $\lambda$. Equation \eqref{elforg} can be
explicitly integrated to show that, in order for $\psi$ to be in
$H^1$, $\lambda$ must be positive and $\psi$ must be a translate of
the function $g_0$ defined in \eqref{g0}, say $\psi(x)=g_0(x+y_0)$ for
some $y_0 \in \mathbb R$. Then \eqref{g0ismin} follows from \eqref{psiismin}.
Also, defining $y_k=\tilde y_k-y_0$, we have that $g_{n_k}(x+y_k)$
converges to $g_0$ in $H^1$.
\end{proof}

\begin{lem}
\label{Is0} Suppose $\beta_1 > 0$, and define $\tilde
J:H^1_{\mathbb C} \to \mathbb R$ by
\begin{equation}
\tilde J(f)=\int_{-\infty}^\infty\left(|f_x|^2 - \beta_1
|f|^{q+2}\right)\ dx.
\end{equation}
Let $s>0$, and let $\{f_n\}$ be any sequence of functions in $H^1_{\mathbb C}$ such that
\begin{equation*}
\lim_{n \to \infty} \|f_n\|^2 = s,
\end{equation*}
 and
\begin{equation*}
\lim_{n \to \infty} \tilde J(f_n) = \inf \ \left\{\tilde J(f): f \in H^1_{\mathbb C}\ {\rm
and }\ \|f\|^2 = s\right\}.
\end{equation*}
Then there exists a subsequence $\{f_{n_k}\}$ of $\{f_n\}$, a sequence of
real numbers $y_k$, and a real number $\theta$ such that $e^{-i\theta}f_{n_k}(x+y_k)$ converges strongly in
$H^1_{\mathbb C}$ norm to $f_0(x)$, where
\begin{equation}
\label{f0}
 f_0(x)=\left(\frac{\lambda}{\beta_1}\right)^{1/q}{\rm
sech}^{2/q}\left(\frac{\sqrt{\lambda} px}{2}\right),
\end{equation}
and $\lambda > 0$ is chosen so that $\|f_0\|^2 = s$. In particular,
\begin{equation}
\tilde J(f_0)=\inf \ \left\{\tilde J(f): f \in H^1_{\mathbb C}\
{\rm and }\ \|f\|^2 = s\right\}. \label{f0ismin}
\end{equation}
\end{lem}

\begin{proof}  The comments in the first paragraph of the proof of Lemma \ref{minforJ} apply as well
to $\tilde J$ as to $J$, since the proof alluded to there works here with no formal changes:  the only
difference is that now $\|f_n\|$ represents the modulus of a complex-valued function.  Therefore we
can conclude that there exists a subsequence $\{f_{n_k}\}$ and a sequence of real numbers $\tilde y_k$ such that
$\{f_{n_k}(x+\tilde y_k)\}$ converges strongly in $H^1_{\mathbb C}$ to a (now complex-valued) function $\phi$
for which
\begin{equation}
\label{phismin}
\tilde J(\phi)=\inf \ \left\{J(f): f \in H^1_{\mathbb C}\ {\rm and }\ \|f\|^2 =
t\right\},
\end{equation}
and for which the Euler-Lagrange equation
\begin{equation}
-2\phi'' -(q+2)\beta_1 \phi^{q+1} = -2\lambda \phi \label{elforphi}
\end{equation}
holds, where here $\lambda$ is again a real number.

It is proved in Theorem 8.1.6 of \cite{[Ca]} that for every solution $\phi$ of $\eqref{elforphi}$, there
exists a real number $\theta$ such that $\phi(x)=e^{i\theta}\tilde \phi(x)$ on $\mathbb R$,
where $\tilde \phi(x)$ is real-valued and positive (the same argument used there is also given below
in the proof of part (iv) of Theorem \ref{existence}).  The $H^1$ function $\tilde \phi$ also satisfies \eqref{elforphi},
and so, as in the proof of Lemma \ref{minforJ}, it follows that there exists $y_0 \in \mathbb R$ such that $\tilde \phi(x)=f_0(x+y_0)$
on $\mathbb R$, where $f_0$ is as defined in \eqref{f0}.
Since $\tilde J(\phi)=\tilde J(\tilde \phi)$, then \eqref{f0ismin} follows from \eqref{phismin}. Also, if we define $y_k =\tilde y_k-y_0$, then we have that
$e^{-i\theta}f_{n_k}(x+y_k)$ converges in $H^1_{\mathbb C}$ to $f_0$.
\end{proof}

\begin{lem} \label{bdfnxbelow}
Suppose $(f_n,g_n)$ is a minimizing sequence for $I(s,t)$, where $s>0$ and $t \ge 0$.
If $t > 0$, or $t=0$ and $\beta_1 >0$, then there exists $\delta>0$ such that $\|f_{nx}\|\geq
\delta$ for all sufficiently large $n$.  If $t=0$ and $\beta_1 = 0$, then $I(s,t)=0$.
\end{lem}

\begin{proof}  As in the proof of Lemma \eqref{bdgnxbelow}, we argue by contradiction.  If the conclusion is not true, then by
passing to a subsequence we may assume there exists a minimizing sequence for which $\displaystyle \lim_{n \to \infty} \|f_{nx}\| = 0$.
From \eqref{GLforf} and \eqref{mixtermbound} we have that
\begin{equation}
\lim_{n \to \infty} \int_{-\infty}^\infty |f_n|^2 g_n = \lim_{n \to \infty} \int_{-\infty}^\infty |f_n|^{q+2} = 0,
\end{equation}
so
\begin{equation}
I(s,t)=\lim_{n \to \infty} \int_{-\infty}^\infty \left(g_{nx}^2 -\beta_2 g_n^{p+2}\right)\ dx.
\label{fgoesaway}
\end{equation}

In case $t>0$, we have from \eqref{g0ismin} that
\begin{equation}
I(s,t) \ge J(g_0),
\label{IgeJ}
\end{equation}
where $g_0$ is as \eqref{g0}, and therefore $g_0$ is integrable with positive integral.  Therefore, by
 Lemma \ref{minfforg} there exists $f \in H^1$ such that $\|f\|^2 =s$ and
\begin{equation}
\int_{-\infty}^\infty \left(f_x^2 - \alpha f^2 g_0\right)\ dx < 0.
\end{equation}
It follows that
\begin{equation}
I(s,t) \le E(f,g_0) =\int_{-\infty}^\infty \left( f_x^2 -\alpha f^2 g_0 -\beta_1 |f|^{q+2}\right)\ dx +J(g_0) < J(g_0),
\end{equation}
which contradicts \eqref{IgeJ}.

In case $t=0$ and $\beta_1 > 0$, then by \eqref{fgoesaway}, $I(s,t)=0$.  On the other hand $I(s,t)=I(s,0)$ is the infimum
of
\begin{equation}
E(f,0)=\int_{-\infty}^\infty\left(|f_x|^2-\beta_1
|f|^{q+2}\right)\ dx \label{efzero}
\end{equation}
over all $f \in H^1_{\mathbb C}$ satisfying $\|f\|^2=s$.  Let $f$
be any non-negative function in $H^1$ such that $\|f\|^2 = s$, and
define $f_\theta(x)=\theta^{1/2}f(\theta x)$.  Then
\begin{equation}
E(f_\theta,0)=\theta^2\int_{-\infty}^\infty f_x^2\ dx - \beta_1 \theta^{q/2}\int_{-\infty}^\infty f^{q+2}\ dx,
\end{equation}
and since $q<4$, we can make the right-hand side negative by choosing a sufficiently small value of $\theta$.  Therefore
$I(s,t)<0$, giving a contradiction.

Finally, if $t=0$ and $\beta_1=0$, then $I(s,t)=I(s,0)$ is the infimum of
\begin{equation}
E(f,0)=\int_{-\infty}^\infty|f_x|^2\ dx
\end{equation}
over all $f$ in $H^1_{\mathbb C}$ such that $\|f\|^2=s$.  This
infimum is clearly non-negative, but on the other hand if we
replace $f$ by $f_\theta$, as defined in the preceding paragraph,
then we can make $E(f_\theta,0)$ arbitrarily small by taking
$\theta$ sufficiently small.  Hence $I(s,t)=0$.
 \end{proof}

\begin{lem} \label{bdmixbelow}
Suppose $(f_n,g_n)$ is a minimizing sequence for $I(s,t)$, where
$s>0$ and $t >0$.  Then there exists $\delta>0$ such that for all
sufficiently large $n$,
\begin{equation*}
\int_{-\infty}^\infty\left(|f_{nx}|^2 -\beta_1|f_n|^{q+2}
-\alpha|f_n|^2g_n\right)\ dx \le -\delta.
\end{equation*}
\end{lem}

\begin{proof} If the conclusion is false, then by passing to a
subsequence we may assume that there exists a minimizing sequence
$(f_n,g_n)$ for which
\begin{equation}
\liminf_{n \to \infty} \int_{-\infty}^\infty\left(|f_{nx}|^2
-\beta_1|f_n|^{q+2} -\alpha|f_n|^2g_n\right)\ dx \ge 0,
\end{equation}
and so
\begin{equation}
I(s,t) = \lim_{n\to\infty}E(f_n,g_n) \ge \liminf_{n \to \infty}
\int_{-\infty}^\infty\left(g_{nx}^2 - \beta_2 g_n^{p+2}\right)\
dx. \label{IgeJg0}
\end{equation}

Define $J$ and $g_0$ as in Lemma \ref{minforJ}.  Then
\eqref{IgeJg0} implies that
\begin{equation}
I(s,t) \ge J(g_0). \label{IgeJg02}
\end{equation} On the other hand, by Lemma
\ref{minfforg}, there exists $f \in H^1$ such that $\|f\|^2=s$ and
\begin{equation*}
\int_{-\infty}^{\infty}\left(f_x^2 -\alpha f^2 g_0\right)\ dx < 0.
\end{equation*}
Therefore
\begin{equation}
I(s,t) \le E(f,g_0) \le \int_{-\infty}^{\infty}\left(f_x^2 -\alpha
f^2 g_0\right)\ dx +J(g_0) < J(g_0),
\end{equation}
which contradicts \eqref{IgeJg02}.
\end{proof}

\begin{lem} \label{posmin}
For all $(f,g)\in Y$, one has $E(|f|,|g|)\leq E(f,g)$.
\end{lem}

\begin{proof}
It is a standard fact from analysis that if $f\in
H_{\mathbb{C}}^{1}$, then $|f(x)|$ is in $H^{1}$ and
\begin{equation}
\int_{-\infty }^{\infty }| |f|_{x}| ^{2}\ dx\leq \int_{-\infty
}^{\infty }|f_{x}|^{2}\ dx.
\label{kinendec}
\end{equation}
(For a proof, the reader may consult Theorem 6.17 of \cite{[LL]}.)
Since $\beta_1$, $\beta_2$, and $\alpha$ are non-negative numbers,
the Lemma follows immediately.
\end{proof}

The next two lemmas state that $E(f,g)$ decreases when $f$ and $g$ are
replaced by $|f|$ and $|g|$, and when $|f|$ and $|g|$ are symmetrically rearranged.  Recall that, for a non-negative
function $w:\mathbb{R} \to [0,\infty)$, if $\{x:w(x)>y\}$ has
finite measure $m(w,y)$ for all $y>0$, then the
symmetric decreasing rearrangement $w^\ast$ of $w$ is defined by
\begin{equation}
w^\ast(x) = \inf\ \{y\in(0,\infty): \frac12 m(w,y) \le x
\} \label{defarr}
\end{equation}
(or see page 80 of \cite{[LL]} for a different but equivalent definition).  For
$(f,g)$ in $Y$, both $|f|$ and $|g|$ are in $H^1$, and hence
$|f|^\ast$ and $|g|^\ast$ are well-defined.

\begin{lem}\label{symmmin}
For all $(f,g)\in Y$, one has $E(|f|^\ast,|g|^\ast)\leq E(f,g)$.
\end{lem}

\begin{proof}
This follows from classic estimates on the symmetric
rearrangements of functions. A basic fact about rearrangements is that they
preserve $L^p$ norms (cf. page 81 of \cite{[LL]}), so that
\begin{equation}
\int_{-\infty }^{\infty }(|f|^\ast) ^{q+2}\ dx = \int_{-\infty
}^{\infty }|f|^{q+2}\ dx
\label{rearrlq}
\end{equation}
and
\begin{equation}
\int_{-\infty }^{\infty }(|g|^\ast) ^{p+2}\ dx = \int_{-\infty
}^{\infty }|g|^{p+2}\ dx.
\label{rearrlp}
\end{equation}
Another basic inequality about rearrangements, Theorem 3.4 of \cite{[LL]},
implies that
\begin{equation}
\int_{-\infty }^{\infty }(|f|^\ast) ^{2}|g|^\ast \ dx\geq
\int_{-\infty }^{\infty }|f|^{2}|g|\ dx.
\label{rearrmix}
\end{equation}
Finally, from Lemma 7.17 of \cite{[LL]} we
have that
\begin{equation*}
\int_{-\infty }^{\infty }| (|f|^\ast)_{x}| ^{2}\ dx\leq
\int_{-\infty }^{\infty }||f|_{x}|^{2}\ dx,
\end{equation*}
and similarly for $g(x)$.
In light of these facts, and because $\alpha$, $\beta_1$, and $\beta_2$ are all non-negative,
it follows from Lemma \ref{posmin} that $E(|f|^\ast,|g|^\ast)\leq E(f,g)$.

\end{proof}

We will also make crucial use of the following Lemma, due to Garrisi \cite{[G]} (see also
the $N$-dimensional version given in Byeon \cite{[By]}).  We
include a proof here since our version of the lemma differs
slightly from that stated by Garrisi.

\begin{lem}  Suppose $u$ and $v$ are non-negative, even, $C^\infty$ functions with compact support in $\mathbb R$, which are non-increasing on $\{x: x \ge
0\}$.  Let $x_1$ and $x_2$ be numbers such that $u(x+x_1)$ and
$v(x+x_2)$ have disjoint supports, and define
$$
w(x)=u(x+x_1)+v(x+x_2).
$$
Let $w^\ast:\mathbb R \to \mathbb R$ be the symmetric decreasing rearrangement
of $w$.  Then the distributional derivative $(w^\ast)'$ of
$w^\ast$ is in $L^2$, and satisfies
\begin{equation}
\|(w^\ast)'\|^2 \le \|w'\|^2 - \frac34 \min\{\|u'\|^2,\|v'\|^2\}.
\label{garineq}
\end{equation}
\label{garlem}
\end{lem}

\begin{proof}  First consider the case when $u'(x) < 0$
for all $x \in (0,c)$ and $v'(x) < 0$ for all $x \in (0,d)$, where
$[-c,c]$ is the support of $u$ and $[-d,d]$ is the support of $v$.
Let $a=\sup\{u(x): x \in \mathbb R\}$ and $b=\sup\{v(x): x \in \mathbb R\}$.  By
interchanging $u$ and $v$ if necessary, we may assume that $a \le
b$.

Define $z_u:[0,\infty) \to [0,c]$ by
\begin{equation}
z_u(y)=\inf\{x \in [0,\infty): u(x) \le y\}.
\label{defyu}
\end{equation}
For $y \in (0,a)$, $z_u(y)$ is equal to the unique
number $x(y) \in (0,c)$ such that $u(x(y))=y$.  The
function $z_u$ is differentiable on $(0,a)$, with derivative
$$
z_u'(y)=\frac{1}{u'(x(y))} < 0,
$$
and we have
\begin{equation*}
\begin{aligned}
\|u'\|^2&=2\int_0^c (u'(x))^2\ dx \\
 &= 2\int_0^a\frac{-1}{z_u'(y)}\ dy\\
 &= 2\int_0^a\frac{1}{|z_u'(y)|}\ dy. \end{aligned}
 \end{equation*}
For $y\ge a$ we have $z_u(y)=0$.

Similarly, we define $z_v:[0,\infty) \to [0,d]$ by
\begin{equation}
z_v(y)=\inf\{x \in [0,\infty): v(x) \le y\}.
\label{defyv}
\end{equation}
Then
$$
y_v'(v(x))=\frac{1}{v'(x)} < 0
$$
on $(0,d)$, and
\begin{equation*}
\|v'\|^2= 2\int_0^b\frac{1}{|z_v'(y)|}\ dy.
 \end{equation*}

Now, for each $y \in [0,\infty)$, define
\begin{equation}
z(y) =  z_u(y)+z_v(y).
\end{equation}
Then $z$ is continuous on $[0,\infty)$ and differentiable, with
strictly negative derivative, on $(0,a)$ and on $(a,b)$. Therefore
$z$ is strictly decreasing on $[0,b]$, and so its restriction to
$[0,b]$ has an inverse function $z^{-1}:[0,c+d]\to [0,b]$, with
$z^{-1}([0,c])=[a,b]$ and $z^{-1}([c,c+d])=([0,a])$. From
\eqref{defarr}  and the definition of $w$, using the fact that
$u(x+x_1)$ and $v(x+x+2)$ have disjoint supports, we see that $w^\ast$ is supported on $[0,c+d]$ and
coincides with $z^{-1}$ there. In particular, for all $y \in
(0,a) \cup (a,b)$, we have
\begin{equation*}
(w^\ast)'(z(y))=\frac{1}{z_u'(y)+z_v'(y)}.
\end{equation*}

Now making use of the fact that for all positive numbers $\mu$ and
$\nu$, there holds the elementary inequality
$$
\frac{2}{\mu + \nu} \le
\frac12\left(\frac{1}{\mu}+\frac{1}{\nu}\right),
$$
we have the following computation:
\begin{equation*}
\begin{aligned}
\|(w^\ast)'\|^2&=2\int_0^{c+d} ((w^\ast)'(x))^2\ dx \\
 &=2 \int_0^c((w^\ast)'(x))^2\ dx +2\int_c^{c+d}((w^\ast)'(x))^2\ dx \\
 &=2 \int_0^a{\frac{1}{|z_u'(y)|+|z_v'(y)|}\ dy}+2\int_a^b{\frac{1}{|z_v'(y)|}\ dy }\\
 &\le \frac12 \int_0^a{\left(\frac{1}{|z_u'(y)|}+\frac{1}{|z_v'(y)|}\right)\ dy}+2\int_a^b{\frac{1}{|z_v'(y)|}\ dy }\\
 &< \frac12 \int_0^a{\frac{1}{|z_u'(y)|}\ dy}+ 2\int_0^a{\frac{1}{|z_v'(y)|}\ dy}+2\int_a^b{\frac{1}{|z_v'(y)|}\ dy }\\
 &=\frac12 \int_0^a{\frac{1}{|z_u'(y)|}\ dy} +2\int_0^b{\frac{1}{|z_v'(y)|}\ dy}\\
&= \frac12\int_0^c{(u'(x))^2\ dx} +2\int_0^d{(v'(x))^2\ dx}\\
&=2\int_0^c{(u'(x))^2\ dx} +2\int_0^d{(v'(x))^2\ dx} -\frac32\int_0^c{(u'(x))^2\ dx}\\
&=\frac12\|u'\|^2+\frac12\|v'\|^2-\frac34\|u'\|^2\\
&=\frac12\|w'\|^2-\frac34\|u'\|^2\\
&\le\frac12\|w'\|^2-\frac34\min\{\|u'\|^2,\|v'\|^2\}.
\end{aligned}
 \end{equation*}
Thus \eqref{garineq} is proved in the special case when $u' <0$ on
$(0,c)$ and $v'<0$ on $(0,d)$.

Now we consider the general case, which we can reduce to the case
treated above as follows.

Let $\phi_1(x)$ be a smooth, even function such that $\phi_1(x)>0$
for $x \in (0,c)$ and $\phi_1(x)=0$ for $x \ge c$, and such that
$\phi_1(x)$ is strictly decreasing on $(0,c)$. Let $\phi_2(x)$ be
a similar function with support on $(0,d)$.  For each $\epsilon >
0$, define $u_\epsilon(x) = u(x)+\epsilon \phi_1(x)$ and
$v_\epsilon(x)=v(x) +\epsilon \phi_2(x)$, and let
$w_\epsilon(x)=u_\epsilon(x)+v_\epsilon(x-T)$.  Since $u' \le 0$
and $\phi_1' <0$ on $(0,c)$, then $u_\epsilon'=u'+\epsilon \phi_1'
<0$ on $(0,c)$, so $u_\epsilon$ is strictly decreasing on $(0,c)$.
Similarly, $v_\epsilon$ is strictly decreasing on $(0,d)$.  So, by
what has been proved above,
\begin{equation}
\|(w_\epsilon^\ast)'\|^2 \le \|w_\epsilon'\|^2 - \frac34
\min\{\|u_\epsilon'\|^2,\|v_\epsilon'\|^2\}. \label{approxineq}
\end{equation}

Now take limits on both sides of \eqref{approxineq} as $\epsilon$
goes to zero.  By the dominated convergence theorem, the right
hand side approaches $$\|w'\|^2 -
\frac34\min\{\|u'\|^2,\|u'\|^2\}.$$ Also, since $w_\epsilon$
converges in $H^1$ norm on $\mathbb R$ to $w$, then by a theorem of Coron
\cite{[Co]}, $w^\ast_\epsilon$ converges in $H^1$ norm to
$w^\ast$.  Therefore the left-hand side of \eqref{garineq}
converges to $\|(w^\ast)'\|^2$, and \eqref{garineq} is proved.
\end{proof}

\begin{lem}\label{Econt}
The functionals $E$, $G$, and $H$ are continuous from $Y$ to
$\mathbb{R}$.
\end{lem}

\begin{proof}
This follows easily (for all $p \ge 0$ and $q \ge 0$) from the
Sobolev embedding theorem, in particular using the fact that the
inclusion of $H^1$ in $L^\infty$ is continuous.
\end{proof}

\begin{lem} \label{subadd}
Let $s_1,s_2,t_1,t_2\geq 0$ be given, and suppose that
$s_1+s_2>0$, $t_1+t_2>0$, $s_1+t_1>0$, and $s_2+t_2>0$. Then
\begin{equation}
\label{SUBA}
I(s_{1}+s_{2},t_{1}+t_{2})<I(s_{1},t_{1})+I(s_{2},t_{2}).
\end{equation}
\end{lem}

\begin{proof}  We claim first that, for $i=1,2$, we can choose minimizing sequences
$(f_n^{(i)},g_n^{(i)})$ for $I(s_i,t_i)$ such that for all $n \in
\mathbb{N}$, $f_n^{(i)}$ and $g_n^{(i)}$
\begin{itemize}
\item{(i)}  are real-valued and non-negative on $\mathbb{R}$;
\item{(ii)} belong to $H^1$ and have compact support;
\item{(iii)} are even functions;
\item{(iv)} are non-increasing functions of $x$ for $x \ge 0$;
\item{(v)} are $C^\infty$ functions; and \item{(vi)} satisfy
$\|f_n^{(i)}\|=s_i$ and $\|g_n^{(i)}\|=t_i$.
\end{itemize}
To prove this, we can take $i=1$, since the proof for $i=2$ is identical.
Also we may assume that $s_1> 0$ and $t_1> 0$, since
otherwise we can simply take $f_n^{(1)}$ or $g_n^{(1)}$ to be
identically zero on $\mathbb{R}$.

Start with an arbitrary
minimizing sequence $(w_n^{(1)},z_n^{(1)})$ for $I(s_1,t_1)$.  Since
 functions with compact support are dense in $H^1$, and
$E: Y \to \mathbb{R}$ is continuous, we can approximate
$(w_n^{(1)},z_n^{(1)})$ by functions $(w_n^{(2)},z_n^{(2)})$ which
have compact support and which still form a minimizing sequence
for $I(s_1,t_1)$.  Then from Lemma \ref{symmmin} it follows that
the sequence defined by
\begin{equation*}
(w_n^{(3)},z_n^{(3)})= (|w_n^{(2)}|^\ast,|z_n^{(2)}|^\ast)
\end{equation*}
is still a minimizing sequence for $I(s_1,t_1)$, and for each $n$,
$w_n^{(3)}$ and $z_n^{(3)}$have the properties (i) through (iv)
listed above.

Next, observe that if $f$ and $\psi$ are any two functions with
properties (i) through (iv), then their convolution $f \star
\psi$, defined as in \eqref{defconvol}, also satisfies properties
(i) through (iv).  Moreover, as is well known, if we define
$\psi_\epsilon = (1/\epsilon)\psi(x/\epsilon)$ for $\epsilon > 0$,
and choose $\psi$ such that $\int_{-\infty}^\infty \psi(x)\ dx =
1$, then convolution with $\psi_\epsilon$ is an ``approximation to
the identity'': that is, the functions $f \star \psi_\epsilon$
converge strongly to $f$ in $H^1$ as $\epsilon \to 0$.  Finally,
if $\psi$ is $C^\infty$ then $f \star \psi_\epsilon$ will be
$C^\infty$ also. Therefore by choosing $\psi(x)$ to be any
non-negative, $C^\infty$, even function with compact support,
which is decreasing for $x \ge 0$, and satisfies
$\int_{-\infty}^\infty \psi(x)\ dx = 1$, and defining
\begin{equation*}
(w_n^{(4)},z_n^{(4)})=
(w_n^{(3)}\star\psi_{\epsilon_n},z_n^{(3)}\star\psi_{\epsilon_n}),
\end{equation*}
with $\epsilon_n$ chosen appropriately small for $n$ large, we obtain a
minimizing sequence $(w_n^{(4)},z_n^{(4)})$ for $I(s_1,t_1)$ that satisfies not only the
properties (i) through (iv) above, but also property (v).

 Finally, we
obtain the desired minimizing sequence satisfying properties (i)
through (vi) by setting
\begin{equation*}
f_n^{(1)}=\frac{(s_{i})^{1/2}w_{n}^{(4)}}{\|w_{n}^{(4)}\|}\text{ \
and \ }g_n^{(1)}=\frac{(t_{i})^{1/2}z_{n}^{(4)}}{\|g_{n}^{(i)}\|},
\end{equation*}
respectively, which is possible since for $n$ sufficiently large
we have $\|w_{n}^{(4)}\|>0$ and $\|z_{n}^{(4)}\| >0$.

Next, choose for each $n$ a number $x_n$
such that $f_n^{(1)}(x)$ and $\tilde f_n^{(2)}(x)= f_n^{(2)}(x+x_n)$ have disjoint support, and
$g_n^{(1)}(x)$ and $\tilde g_n^{(2)}(x)=g_n^{(2)}(x+x_n)$ have disjoint support.  Define
\begin{equation*}
\begin{aligned}
f_n &= \left(f_n^{(1)} + \tilde f_n^{(2)}\right)^\ast,\\
g_n &= \left(g_n^{(1)} + \tilde g_n^{(2)}\right)^\ast.
\end{aligned}
\end{equation*}
Then $\|f_n\|^2 = s_1 + s_2$ and $\|g_n\|^2 = t_1 + t_2$, so
\begin{equation}
I(s_1+s_2,t_1 + t_2) \le E(f_n,g_n).
\label{IleE}
\end{equation}
On the other hand, from Lemma \ref{garlem} we have that
\begin{equation}
\begin{aligned}
\int_{-\infty}^\infty &{\left( f_{nx}^2 + g_{nx}^2 \right)\ dx}
\le
\int_{-\infty}^\infty {\left( (f_n^{(1)}+\tilde f_n^{(2)})_x^2 + (g_n^{(1)}+\tilde g_n^{(2)})_x^2 \right)\ dx} - K_n\\
 &=\int_{-\infty}^\infty {\left( (f_{nx}^{(1)})^2+(\tilde f_{nx}^{(2)})^2 + (g_{nx}^{(1)})^2+(\tilde g_{nx}^{(2)})^2 \right)\ dx} - K_n,
\end{aligned}
\label{kinenstrictdec}
\end{equation}
where
\begin{equation}\label{defKn}
K_n = \frac34
\left(\min\left\{\|f_{nx}^{(1)}\|^2,\|f_{nx}^{(2)}\|^2\right\}
+\min\left\{\|g_{nx}^{(1)}\|^2,\|g_{nx}^{(2)}\|^2\right\}\right).
\end{equation}

Furthermore, from the properties \eqref{rearrlp}, \eqref{rearrlq}, and \eqref{rearrmix} of rearrangements, we have that
\begin{equation}
\begin{aligned}
\int_{-\infty}^\infty |f_n|^{q+2}\ dx &= \int_{-\infty}^\infty |f_n^{(1)}|^{q+2}\ dx + \int_{-\infty}^\infty |f_n^{(2)}|^{q+2}\ dx\\
\int_{-\infty}^\infty g_n^{p+2}\ dx &= \int_{-\infty}^\infty (g_n^{(1)})^{p+2}\ dx + \int_{-\infty}^\infty (g_n^{(2)})^{q+2}\ dx\\
\int_{-\infty}^\infty |f_n|^2g_n\ dx &\ge \int_{-\infty}^\infty |f_n^{(1)}|^2g_n^{(1)}\ dx + \int_{-\infty}^\infty |f_n^{(2)}|^2g_n^{(2)}\ dx,\\
\end{aligned}
\end{equation}
and therefore, combining with \eqref{IleE} and \eqref{kinenstrictdec},
we have that for every $n$,
\begin{equation}
I(s_1+t_1, s_2+t_2) \le E(f_n,g_n) \le E(f_n^{(1)},g_n^{(1)}) + E(f_n^{(2)},g_n^{(2)}) - K_n.
\end{equation}
It follows by taking the limit superior on the right-hand side that
\begin{equation}
I(s_1+t_1,s_2+t_2) \le I(s_1,t_1)+I(s_2,t_2) - \liminf_{n \to \infty} K_n.
\label{subaddKn}
\end{equation}

Since $t_1 + t_2 >0$, then either $t_1$ and $t_2$ are both
positive,  or one of $t_1$ and $t_2$ is zero and the other is
positive. In the latter case, we may assume that $t_1 = 0$ and
$t_2 >0$, since otherwise we can simply switch $t_1$ and $t_2$.
Then we will argue separately according as to whether $s_2$ is
positive or zero.  To prove the theorem, then, it suffices to
consider the following three cases: (i) $t_1
> 0$ and $t_2>0$; (ii) $t_1 = 0$, $t_2 > 0$, and $s_2 > 0$; and (iii)
$t_1 = 0$, $t_2 >0$, and $s_2=0$.

In case (i), when $t_1 > 0$ and $t_2 >0$, it follows from Lemma
\ref{bdgnxbelow} that there exist numbers $\delta_1 >0$ and
$\delta_2 >0$ such that for all sufficiently large $n$,
$\|(g_n^{(1)})_x\| \ge \delta_1$ and $\|(g_n^{(2)})_x\| \ge
\delta_2$. (Note that by Lemma \ref{bdfnxbelow}, this is still
true even when $s_1 = 0$ or $s_2 = 0$.) So, letting $\delta =
\min(\delta_1,\delta_2)> 0$, \eqref{defKn} gives $K_n \ge
3\delta/4$ for all sufficiently large $n$. From \eqref{subaddKn}
we then have that
\begin{equation}
I(s_1+t_1,s_2+t_2) \le I(s_1,t_1)+I(s_2,t_2) - 3\delta/4 <
I(s_1,t_1)+I(s_2,t_2), \label{subaddwdelta}
\end{equation} as
desired.

In case (ii), we have $t_1=0$, $t_2 > 0$, $s_2 >0$, and, since
$s_1+t_1 >0$ by assumption, $s_1 >0$ also.  By Lemma
\ref{bdfnxbelow} there exists $\delta_1 > 0$ such that for all
sufficiently large $n$, $\|(f_n^{(1)})_x\| \ge \delta_1$.

If, in case (ii), $\beta_1> 0$, then by Lemma \ref{bdfnxbelow}
there also exists $\delta_2
>0$ such that for all sufficiently large $n$,
$\|f_{nx}^{(2)}\| \ge \delta_2$.  Letting $\delta =
\min(\delta_1,\delta_2)> 0$, we get $K_n \ge 3\delta/4$ for large
$n$, and \eqref{subaddwdelta} follows from \eqref{subaddKn} as in
case (i).

On the other hand, if in case (ii) we have $\beta_1 =0$, then by
Lemma \ref{bdfnxbelow} we have $I(s_1,t_1)=I(s_1,0)=0$, and
$I(s_1+s_2,t_1+t_2)=I(s_1+s_2,t_2)$ is the infimum of
\begin{equation}
E(f,g)=\int_{-\infty}^\infty \left(|f_x|^2 + g_x^2 -\beta_2
g^{p+2} -\alpha |f|^2 g\right)\ dx \label{Ebeta10}
\end{equation}
over all $f \in H^1_{\mathbb C}$ and $g \in H^1$ such that
$\|f\|^2 = s_1+s_2$ and $\|g\|^2=t_2$.  By Lemma \ref{bdmixbelow},
there exists $\delta >0$ such that for all sufficiently large $n$,
\begin{equation*}
 \int_{-\infty}^\infty\left(|f_{nx}^{(2)}|^2
-\alpha|f_n^{(2)}|^2g_n^{(2)}\right)\ dx \le -\delta.
\end{equation*}
Let
\begin{equation}
f_n = \sqrt{\frac{s_1 + s_2}{s_2}} f_n^{(2)};
\end{equation}
then $\|f_n\|^2 = s_1+s_2$ and from \eqref{Ebeta10} we see that,
for all sufficiently large $n$,
\begin{equation}
\begin{aligned}
I(s_1+s_2,t_2) \le E(f_n,g_n^{(2)}) &= E(f_n^{(2)},g_n^{(2)}) +
\frac{s_1}{s_2}
 \int_{-\infty}^\infty\left(|f_{nx}^{(2)}|^2
-\alpha|f_n^{(2)}|^2g_n^{(2)}\right)\ dx \\
\le E(f_n^{(2)},g_n^{(2)}) -\frac{s_1\delta}{s_2}.
\end{aligned}
\end{equation}
This implies, after taking the limit as $n \to \infty$, that
\begin{equation}
I(s_1+s_2,t_2) \le I(s_2,t_2) -\frac{s_1\delta}{s_2} <
I(s_2,t_2)=I(s_1,t_1)+I(s_2,t_2),
\end{equation}
as desired.  Thus the proof is complete in case (ii).

In case (iii), we have $s_1>0$ and $t_2>0$, and we have to prove
\begin{equation}
I(s_1,t_2) < I(s_1,0)+I(0,t_2). \label{caseiii}
\end{equation}
Let $g_0$ be as defined in Lemma \ref{minforJ} with $t=t_2$, so
that $I(0,t_2)=J(g_0)$.

 If $\beta_1>0$, we have from Lemma \ref{Is0} that
$I(s_1,0)=\tilde J(f_0)$, where $f_0$ is as defined in \eqref{f0}
with $s=s_1$. Clearly,
\begin{equation*}
\int_{-\infty}^\infty |f_0|^2 g_0\ dx >0,
\end{equation*}
and so
\begin{equation}
\begin{aligned}
 I(s_1,t_2) \le E(f_0,g_0) &= \tilde J(f_0) + J(g_0) +
\int_{-\infty}^\infty |f_0|^2 g_0\ dx  \\
  & < \tilde J(f_0) + J(g_0) = I(s_1,0)+I(0,t_2),
  \end{aligned}
  \end{equation}
  as desired.

On the other hand, if $\beta_1 = 0$, then  $I(s_1,0)=0$ by Lemma
\ref{bdfnxbelow}.  By Lemma \ref{minfforg}, there exists $f \in
H^1$ such that $\|f\|^2 = s_1$ and
\begin{equation}
\int_{-\infty}^\infty \left(f_x^2 - \alpha f^2 g_0\right)\ dx < 0,
\end{equation}
and hence
\begin{equation}
I(s_1,t_1) \le E(f,g_0) = \int_{-\infty}^\infty \left(f_x^2 -
\alpha f^2 g_0\right)\ dx + J(g_0) < J(g_0),
\end{equation}
which proves \eqref{caseiii}.  The proof of Lemma \ref{subadd} is
now complete in all cases.
\end{proof}

We now turn to the proof of Theorem \ref{existence}, which, once
the subadditivity lemma \ref{subadd} has been established,
proceeds by largely the same argument as in \cite{[AA]}.

The first step is to establish the relative compactness, up to
translations, of minimizing sequences for $I(s,t)$. Let
$\{(f_{n},g_{n})\}$ be a given minimizing sequence, and define an
associated sequence of functions $\rho_n$ by
\begin{equation*}
\rho_n=|f_n|^2+g_n^2.
\end{equation*}
We then have
\begin{equation*}
\int_{-\infty }^{\infty }\rho _{n}(x)\ dx = s+t
\end{equation*}
for all $n$. The sequence of functions $M_{n}:[0,\infty)\to
[0,s+t]$ defined by
\begin{equation*}
M_n(r)=\sup_{y\in \mathbb{R}}\int_{y-r}^{y+r}\rho_n(x)\ dx.
\end{equation*}
is a uniformly bounded sequence of nondecreasing functions on
$[0,\infty)$, and therefore (by Helly's selection theorem, for
example) has a subsequence, which we will still denote by $M_n$,
that converges pointwise to a nondecreasing function $M$
on $[0,\infty)$.  Then
\begin{equation}
\label{defgamma} \gamma =\lim_{r\to \infty }M(r)
\end{equation}
exists and satisfies $0\leq \gamma \leq s+t$.

We claim now that $\gamma > 0$.  To prove this, we require the following
lemma.

\begin{lem}  \label{Lionsvanish}
Suppose $w_n$ is a sequence of functions which is bounded in $H^1$
and which satisfies, for some $R >0$,
\begin{equation}
\lim_{n \to \infty} \sup_{y \in \mathbb R} \int_{y-R}^{y+R} w_n^2\ dx = 0.
\label{vanishhypo}
\end{equation}
Then for every $r>2$,
\begin{equation*}
\lim_{n \to \infty} |w_n|_r = 0.
\end{equation*}
\end{lem}

\begin{proof} This is a special case of Lemma I.1 of part 2 of \cite{[L]}, but for the sake
of completeness we give a proof here.  Let
\begin{equation}
\epsilon_n = \sup_{y \in \mathbb R} \int_{y-R}^{y+R} w_n^2\ dx,
\end{equation}
so that $\displaystyle \lim_{n \to \infty} \epsilon_n = 0$.
For every $y \in \mathbb R$, we have by standard Sobolev inequalities
(see Theorem 10.1 of \cite{[F]}) that
\begin{equation*}
\int_{y-R}^{y+R} |w_n|^r\ dx \le C \left(\int_{y-R}^{y+R} |w_n|^2\ dx\right)^s\left(\int_{y-R}^{y+R}\left(w_n^2+w_{nx}^2\right)\ dx\right)^{1+s},
\end{equation*}
where $s=(r-2)/4$.  It then follows from \eqref{vanishhypo} that
\begin{equation}\label{piecesofR}
\begin{aligned}
\int_{y-R}^{y+R} |w_n|^r\ dx  &\le C\epsilon_n^s\left(\int_{y-R}^{y+R}\left(w_n^2+w_{nx}^2\right)\ dx\right)\|w_n\|_1^s \\
& \le C \epsilon^s\int_{y-R}^{y+R}\left(w_n^2+w_{nx}^2\right)\ dx.\\
\end{aligned}
\end{equation}
Now if we cover
$\mathbb R$ by intervals of length $R$ in such a way that each point of $\mathbb R$ is contained in at most two
of the intervals, then by summing \eqref{piecesofR} over all the intervals in the cover, we obtain that
\begin{equation*}
|w_n|_r \le 3C \epsilon_n^s\|w_n\|_1^2 \le C \epsilon_n^s,
\end{equation*}
from which the desired result follows.
\end{proof}

Next we prove that
\begin{equation}
\label{gamnezero}
 \gamma \ne 0.
\end{equation}
Indeed, suppose for the sake of contradiction that $\gamma =0$.
Then \eqref{vanishhypo} holds both for $w_n=|f_n|$ and for
$w_n=g_n$.  Since both $\{|f_n|\}$ and $\{g_n\}$ are bounded
sequences in $H^1$ by Lemma \ref{Ibounded}, then Lemma
\eqref{Lionsvanish} implies that for every $r
> 2$, $f_n$ and $g_n$ converge to 0 in $L^r$ norm.  Since
\begin{equation*}
\left\vert \int_{-\infty }^{\infty }|f_n|^2g_n\ dx\right\vert \leq
|f_n|_4^{1/2}\|g_n\|
\end{equation*}
and $\|g_{n}\|$ is bounded, it follows also that
\begin{equation*}
\lim_{n \to \infty} \int_{-\infty }^{\infty }|f_n|^2g_n\ dx =0.
\end{equation*}
Hence
\begin{equation}
I(s,t)=\lim_{n\to \infty }E(f_n,g_n) \ge \liminf_{n \to
\infty}\int_{-\infty }^{\infty }\left( |f_{nx}|^2+g_{nx}^2\right)
dx\geq 0,
\label{istgezero}
\end{equation}
contradicting Lemma~\ref{Ibounded}. This proves \eqref{gamnezero}.

\begin{lem} \label{DIC} Suppose $\gamma$ is defined as in \eqref{defgamma}.  Then there exist numbers $s_1 \in [0,s]$ and $t_1 \in
[0,t]$ such that
\begin{equation}
\label{SUM} \gamma = s_1+t_1
\end{equation}
and
\begin{equation} \label{REV}
I(s_{1},t_{1})+I(s-s_{1},t-t_{1})\leq I(s,t).
\end{equation}
\end{lem}

\begin{proof}  Since the proof is almost the same as the proof of Lemma 3.10 of \cite{[AA]}, with only
slight modifications, we just give an outline here, and refer to \cite{[AA]} for the details.  Let $\rho$ and
$\sigma$ be smooth functions on $\mathbb R$ such that $\rho^2 + \sigma^2 = 1$ on $\mathbb R$, and $\rho$ is identically 1
on $[-1,1]$ and is supported in $[-2,2]$; and define $\rho_\omega(x)=\rho(x/\omega)$ and $\sigma_\omega(x)=\sigma(x/\omega)$
for $\omega > 0$.  From the definition of $\gamma$ it follows that for given $\epsilon >0$, there exist $\omega > 0$
and a sequence $y_n$ such that, after passing to a subsequence, the functions $(f_n^{(1)}(x), g_n^{(1)}(x))=\rho_\omega(x-y_n)(f_n(x),g_n(x))$
and $(f_n^{(2)}(x), g_n^{(2)}(x))=\sigma_\omega(x-y_n)(f_n(x),g_n(x))$ satisfy
$\|f_n^{(1)}\|^2 \to s_1$, $\|g_n^{(1)}\|^2 \to t_1$, $\|f_n^{(2)}\|^2 \to s-s_1$, and $\|g_n^{(2)}\|^2 \to t-t_1$ as $n \to \infty$,
where $|(s_1+t_1)-\alpha|<\epsilon$, and
\begin{equation} \label{e12ineq}
E(f_n^{(1)},g_n^{(1)})+E(f_n^{(2)},g_n^{(2)}) \le E(f_n,g_n) + C\epsilon
\end{equation}
for all $n$.  To prove \eqref{e12ineq}, one writes
\begin{equation*}
\begin{aligned}
&E(f_n^{(1)},g_n^{(1)})= \int_{-\infty}^\infty
\rho_\omega^2\left(|f_{nx}|^2+g_{nx}^2-\beta_1 |f_n|^{q+2}-\beta_2
g_n^{p+2}-\alpha|f_n|^2g_n\right)\ dx\\
&\ \ +\int_{-\infty}^\infty\left((\rho_\omega^2-\rho_\omega^{q+2})\beta_1|f_n|^{q+2}
+(\rho_\omega^2-\rho_\omega^{p+2})\beta_2|g_n|^{p+2}
+(\rho_\omega^2-\rho_\omega^3)\alpha|f_n|^2g_n\right)\ dx\\
&\ \   +\int_{-\infty}^\infty\left((\rho_\omega')^2(|f_{nx}|^2 + g_{nx}^2)+2\rho_\omega\rho_\omega'(\text{Re}\ f_n(\overline{f_{n}})_x + g_ng_{nx}
)+(\rho_\omega')^2|f|^2\right)\ dx,
\end{aligned}
\end{equation*}
and observes that the last two integrals on the right hand side can be made arbitrarily uniformly small by taking $\omega$ sufficiently large.
A similar estimate obtains for $E(f_n^{(2)},g_n^{(2)})$, and \eqref{e12ineq} follows by adding the two estimates and using $\rho_\omega^2 + \sigma_\omega^2 = 1$.

  Now we show that the limit inferior as $n \to \infty$ of the
left-hand side of \eqref{e12ineq} is greater than or equal to $I(s_1,t_1) + I(s-s_1,t-t_1)$.  If $s_1$, $t_1$, $s-s_1$, and $t-t_1$ are all positive,
this follows by rescaling $f_n^{(i)}$ and $g_n^{(i)}$ for $i=1,2$ so that
$\|f_n^{(1)}\|^2 = s_1$, $\|g_n^{(1)}\|^2 = t_1$, $\|f_n^{(2)}\|^2 = s-s_1$, and $\|g_n^{(2)}\|^2 = t-t_1$, since the scaling factors
tend to 1 as $n \to \infty$.  On the other hand, if $s_1=0$ and $t_1>0$ then as in \eqref{istgezero} we have
\begin{equation*}
\begin{aligned}
\lim_{n \to \infty}E(f_n^{(1)},g_n^{(1)})&=\lim_{n \to
\infty}\int_{-\infty}^\infty\left(|f_{nx}^{(1)}|^2+(g_{nx}^{(1)})^2
-\beta_2(g_n^{(1)})^{q+2}\right)\ dx\\
&\ge \liminf_{n\to \infty}\int_{-\infty}^\infty((g_{nx}^{(1)})^2
-\beta_2(g_n^{(1)})^{q+2})\ dx \ge I(0,t_1),
\end{aligned}
\end{equation*}
and similar estimates hold if $t_1$, $s-s_1$, or $t-t_1$ are zero.

Taking then the limit inferior of the left-hand side and the limit of the right-hand side of \eqref{e12ineq} as $n \to \infty$,
we obtain
\begin{equation*}
I(s_{1},t_{1})+I(s-s_{1},t-t_{1})\leq I(s,t) + C\epsilon,
\end{equation*}
which proves \eqref{REV}, as $\epsilon$ is arbitrary.
\end{proof}

We claim now that
\begin{equation}
\gamma = s+t. \label{gameqspt}
\end{equation}
Suppose to the contrary that $\gamma < s + t$.   Let $s_1$ and
$t_1$ be as defined in Lemma~\ref{DIC}, and let $s_2=s-s_1$ and
$t_2=t-t_1$. Then $s_{2}+t_{2}=(s+t)-\gamma >0$, and also
\eqref{gamnezero} and \eqref{SUM} imply that $s_{1}+t_{1}>0$.
Moreover, $s_{1}+s_{2}=s>0$ and $t_{1}+t_{2}=t>0$. Therefore
Lemma~\ref{subadd} implies that that \eqref{SUBA} holds. But this
contradicts \eqref{REV}.  Thus \eqref{gameqspt} is proved.

Once \eqref{gameqspt} has been established, assertion (ii) of
Theorem \ref{existence}, concerning the relative compactness of
minimizing sequences up to translation, follows from general
principles.  We again only outline the proof here and refer the
reader to \cite{[AA]} for more details.  First, \eqref{gameqspt}
immediately implies that for some sequence $y_n$ of real numbers
and some fixed subsequence of $(f_n,g_n)$, denoted again by
$(f_n,g_n)$, and for every $k \in \mathbb{N}$, there exists
$\omega _{k}\in \mathbb{R}$ such that
\begin{equation} \label{EJ10}
\int_{-\omega_k}^{\omega_k}\left(
|f_n(x+y_n)|^2+g_n(x+y_n)^2\right) \ dx \ge s+t-\frac{1}{k}.
\end{equation}
for all $n \in \mathbb N$.  (In other words, the measures
\begin{equation*}
\mu_n = (|f_n(x+y_n)|^2+g_n(x+y_n)^2)\ dx
\end{equation*}
 form a
``tight'' family on $\mathbb R$, in the sense that for every
$\epsilon >0$, there exists a fixed compact set $K$ such that
$\mu_n(\mathbb R \backslash K) < \epsilon$ for all $n \in \mathbb
N$.)

From \eqref{EJ10} and the compactness of the embedding of $H^1$
into $L^2$ on finite domains, it follows that some further
subsequence of $(f_n(x+y_n),g(x+y_n))$ converges, strongly in
$L^2(\mathbb R)\times L^2(\mathbb R)$ and weakly in $Y$, to a
limit $(\phi,\psi)$. Estimates such as \eqref{GLforf} and
\eqref{mixtermbound}, together with the weak lower semicontinuity
of the Hilbert space norm in $Y$, imply that
\begin{equation}
\lim_{n \to \infty} E(f_n,g_n) \ge E(\phi,\psi);
\end{equation}
but since $(f_n,g_n)$ is a minimizing sequence, this in turn
implies that
\begin{equation}
\lim_{n \to \infty} E(f_n,g_n) = E(\phi,\psi).
\end{equation}
Therefore one has
\begin{equation*}
\lim_{n\to \infty }\int_{-\infty }^{\infty }\left(|f_{nx}|
^2+g_{nx}^2\right) \ dx=\int_{-\infty }^{\infty }\left(
|\phi_x|^2+\psi_x^2\right) \ dx,
\end{equation*}
so $(f_n(x+y_n),g_n(x+y_n))$ converges strongly to $(\phi,\psi)$
in the norm of $Y$.

Since $(\phi,\psi)$ is in the minimizing set $\mathcal S_{s,t}$
for $I(s,t)$, and so minimizes $E(u,v)$ subject to $H(u)$ and
$H(v)$ being held constant, the Lagrange multiplier principle
(see, for example, Theorem 7.7.2 of \cite{[Lu]}) asserts that
there exist real numbers $\sigma$ and $c$ such that
\begin{equation}
\delta E(\phi,\psi) = \sigma \delta H(\phi)+c \delta H(\psi),
\end{equation}
where $\delta$ denotes the Fr\'echet derivative.  Computing the
Fr\'echet derivatives we see that this means that equations
\eqref{ODE} hold, at least in the sense of distributions.  But since the right-hand
sides of the equations in \eqref{ODE} are continuous functions of the unknowns, distributional
solutions are also classical solutions (cf.\ Lemma 1.3 of \cite{[T]}). This
then proves assertion (iii) of Theorem \ref{existence}.

It remains to prove the assertions in part (iv) of Theorem
\ref{existence}.

Multiplying the first equation in \eqref{ODE} by $\overline \phi$
and integrating over $\mathbb R$, we have after an integration by
parts that
\begin{equation}
\label{intforsigma} \int_{-\infty}^\infty \left(|\phi'|^2 - \tau_1
|\phi|^{q+2} - \alpha |\phi|^2\psi\right)\ dx = -\sigma
\int_{-\infty}^\infty |\phi|^2\ dx = -\sigma s.
\end{equation}
In particular, it follows from \eqref{intforsigma} that $\sigma$
is real. Similarly, multiplying the second equation in \eqref{ODE}
by $\psi$ and integrating over $\mathbb R$ yields
\begin{equation}
\label{intforc} \int_{-\infty}^\infty \left(|\psi'|^2 -
\frac{\tau_2}{p+1} \psi^{p+2} - \frac{\alpha}{2}
|\phi|^2\psi\right)\ dx = -c \int_{-\infty}^\infty |\psi|^2\ dx =
-ct.
\end{equation}

From Lemma \ref{bdmixbelow}, applied to the constant sequence
$(f_n,g_n)=(\phi,\psi)$, we have that
\begin{equation}
\label{firstint} \int_{-\infty}^\infty \left(|\phi'|^2 - \tau_1
|\phi|^{q+2} - \alpha |\phi|^2\psi\right)\ dx  < 0,
\end{equation}
and since $\tau_1 = \beta_1(q+2)/2 > \beta_1$, it follows that the
integral on the left-hand side of \eqref{intforsigma} is negative,
 and so we must have $\sigma > 0$.  Therefore, a calculation with the Fourier transform shows that
 the operator $-\partial_x^2 + \sigma$ appearing in the first equation of \eqref{ODE} is invertible on $H^1_{\mathbb C}$, with inverse given
by convolution with the function
\begin{equation*}
K_\sigma (x)= \frac{1}{2\sqrt{\sigma}}e^{-\sqrt \sigma |x|}.
\end{equation*}
The first equation of \eqref{ODE} can then be rewritten in the
form
\begin{equation}
\label{convforphi} \phi=K_\sigma \star\left(\tau_1|\phi|^q\phi +
\alpha \phi\psi\right),
\end{equation}
where $\star$ denotes convolution as in \eqref{defconvol}.

Now we observe that it follows from the first equation in
\eqref{ODE} that there exist $\theta \in \mathbb R$ and a
real-valued function $\tilde \phi(x)$ such that
$\phi(x)=e^{i\theta}\tilde\phi(x)$ on $\mathbb R$.  This is proved
for the case $\tau_1=0$ in part (iii) of Theorem 2.1 of
\cite{[AA]}, and it is easy to check that the same proof works as
well when $\tau_1 \ne 0$.

Note next that $(\tilde\phi,|\psi|)$ and $(|\tilde\phi|,|\psi|)$
are also in $\mathcal S_{s,t}$, as follows from Lemma
\eqref{posmin}.  Therefore, if we let $w = |\tilde \phi|$, then
$\tilde \phi$ and $w$ satisfy the Lagrange multiplier equations
\begin{equation}
\begin{aligned}
-\tilde\phi''+\sigma \tilde \phi & =\tau_1w^q\tilde\phi+\alpha \tilde\phi|\psi| \\
-w''+\sigma w & =\tau_1w^qw+\alpha w|\psi|.
\end{aligned}
\end{equation}
(That the Lagrange multiplier $\sigma$ is the same in both
equations follows from the fact that $\sigma$ is determined by the
equation \eqref{intforsigma}, and this equation is unchanged when
$\phi$ is replaced by $w$.) Multiplying the first equation by $w$
and the second equation by $\tilde \phi$, and subtracting the two
equations, we find that the $w\tilde\phi'' - \tilde \phi w'' = 0$.
Therefore the Wronskian $w\tilde\phi' - \tilde \phi w'$ of $w$ and
$\tilde \phi$ is constant, and since $w$ and $\tilde\phi$ are both
in $H^1$, this constant must be zero.  So $w$ and $\tilde \phi$
are constant multiples of each other, and hence $\tilde \phi$,
like $w$, must be of one sign on $\mathbb R$.  By replacing
$\theta$ by $\theta + i\pi$ if necessary, we can assume that
$\tilde \phi \ge 0$ on $\mathbb R$.

We claim that
\begin{equation}
\label{psiclaim} \int_{-\infty}^\infty|\phi|^2|\psi| \
dx=\int_{-\infty}^\infty|\phi|^2\psi \ dx.
\end{equation}
To prove this, note that since
$E(|\phi|,|\psi|)=E(|\phi|,\psi)=I(s,t)$, we have
\begin{equation}
\alpha \int_{-\infty}^\infty |\phi|^2(|\psi|-\psi)\ dx=
\int_{-\infty}^\infty \left((|\psi_x|^2-\psi_x^2)
-\beta_2(|\psi|^{p+2}-\psi^{p+2})\right)\ dx.
\end{equation}
Using \eqref{kinendec}, we see that the right-hand side of this
equation is less than or equal to zero, so we must have
\begin{equation}
\alpha \int_{-\infty}^\infty |\phi|^2(|\psi|-\psi)\ dx \le 0
\end{equation}
also.  But since the integrand is non-negative, this proves
\eqref{psiclaim}.

From \eqref{psiclaim} it follows that $\psi(x) \ge
0$ at every point $x$ in $\mathbb R$ for which $\tilde \phi(x) \ne
0$. Now \eqref{convforphi} implies that
\begin{equation}
\label{convfortildephi} \tilde\phi=K_\sigma
\star\left(\tau_1|\tilde \phi|^q\tilde\phi + \alpha \tilde
\phi\psi\right).
\end{equation}
Since the convolution of $K_\sigma$ with a function that is
 everywhere non-negative and not
identically zero must produce an everywhere positive function, it
follows that $\tilde \phi(x)>0$ for
all $x \in \mathbb R$.  But this in turn implies that
$\psi(x) \ge 0$ for all $x \in \mathbb R$.

Now suppose, for the sake of contradiction, that $\psi(x_0)=0$ for some $x_0 \in \mathbb R$.
Then from the preceding paragraph it follows that $x_0$ is a point where $\psi$ takes its minimum
value over $\mathbb R$, and therefore we must have $\psi'(x_0)=0$.  But then standard uniqueness theory
for ordinary differential equations, applied to the second equation in \eqref{ODE} viewed as an inhomogeneous equation
for $\psi$, yields that $\psi$ must be identically zero on its entire interval of existence about $x_0$, which in
this case is $\mathbb R$.  But this contradicts the fact that $\|\psi\|^2 = t >0$.  Therefore $\psi$ must
be everywhere positive on $\mathbb R$.

Finally, since $\psi$ and $|\phi|$ are everywhere positive on $\mathbb R$, and the right-hand sides
of the equations in \eqref{ODE} are infinitely differentiable functions of $\phi$ and $\psi$ on the domain
$\{(\phi,\psi)\in \mathbb C \times \mathbb R: |\phi|>0 \ \text{and}\ \psi >0\}$, it follows from the standard
theory of ordinary differential equations that any solution of \eqref{ODE} must be infinitely differentiable
on its interval of existence.

This completes the proof of Theorem \ref{existence}.

\section{Stability of Solitary-Wave Solutions}
\label{sec:stability}

In this section we prove the stability result given in Theorem
\ref{stability} by considering the variational problem $W(s,t)$
defined in \eqref{Wst}.   In particular, we show that arbitrary
minimizing sequences for $W(s,t)$ converge, up to subsequences and
translations, to elements of the minimizing set
$\mathcal{F}_{s,t}$ defined in \eqref{Fst}.  To do this, we relate
the problem in \eqref{Wst} to that in \eqref{Ist}, following the
method used in \cite{[AA]}.

\begin{lem}\label{bounded}  Suppose $1 \le q < 4$ and $1 \le p <
4/3$, and let $s>0$ and $t \in \mathbb{R}$. If $\{(h_{n},g_{n})\}$
is a minimizing sequence for $W(s,t)$, then $\{(h_{n},g_{n})\}$ is
bounded in $Y$.
\end{lem}

\begin{proof} Since $\|h_{n}\|^2=H(h_{n})$ is bounded, then
\begin{equation}
\begin{aligned}
\|g_n\|^2 = \left|G(h_n,g_n)-\text{Im}\int_{-\infty }^{\infty
}h_n(\overline{h_n})_x\ dx\right| &\leq C\left( 1+\|h_n\|
\cdot\|h_{nx}\| \right) \\
&\leq C\left( 1+\|(h_n,g_n)\|_Y\right),
\end{aligned}
\label{gnbound}
\end{equation}
where $C$ is independent of $n$.  Therefore
\begin{equation}
\begin{aligned}
\label{bdynorm} \|(h_{n},g_{n})\|_{Y}^{2}
&=E(h_{n},g_{n})+\int_{-\infty }^{\infty
}\left(\beta_1|h_n|^{q+2}+\beta_2 g_n^{p+2}+ \alpha
|h_n|^2g_{n}\right)\ dx
+\|h_n\|^2+\|g_n\|^2\\
&\le C\int_{-\infty}^\infty \left(|h_n|^{q+2} +|g_n|^{p+2} +
|h_n|^2|g_n|\right)\ dx +C\left(1+\|(h_n,g_n)\|_Y\right).
\end{aligned}
\end{equation}

From \eqref{gnbound} it follows that
\begin{equation*}
\begin{aligned}
\int_{-\infty }^{\infty }|g_n|^{p+2}\ dx &\leq
C\|g_{nx}\|^{p/2}\|g_n\|^{(p+4)/2} \\
&  \leq
C\left(\|(h_n,g_n)\|_Y^{p/2}+\|(h_n,g_n)\|_Y^{(3p+4)/4}\right).
\end{aligned}
\end{equation*}
On the other hand, as in \eqref{fLqbound}, we have
\begin{equation*}
\int_{-\infty }^{\infty }|h_n|^{q+2}\ dx
  \leq C\|h_{nx}\|^{q/2} \|h_n\|^{(q+4)/2}
  \le C\|(h_n,g_n)\|_Y^{q/2},
\end{equation*}
and, as in \eqref{mixtermbound},
\begin{equation*}
\int_{-\infty }^{\infty } |h_n|^{2}|g_{n}|\ dx\leq
C\|h_{nx}\|^{1/2}\|g_n\| \leq C\left( 1+\|(h_n,g_n)\|_Y\right).
\end{equation*}
Combining these estimates with \eqref{bdynorm} gives
\begin{equation*}
\begin{aligned}
&\|(h_n,g_n)\|_Y^2 \\
&\ \ \le C\left(1 +
\|(h_n,g_n)\|_Y+\|(h_n,g_n)\|_Y^{q/2}+\|(h_n,g_n)\|_Y^{p/2}+\|(h_n,g_n)\|_Y^{(3p+4)/4}\right),
\end{aligned}
\end{equation*}
and since $q<4$ and $p<4/3$, the exponents on the right-hand side
are all less than 2.  Hence $\|(h_n,g_n)\|_Y$ is bounded.
\end{proof}

We omit the proofs of the next two lemmas, which are identical to
the proofs of Lemmas 4.2 and 4.3 in \cite{[AA]}.

\begin{lem}\label{decom}
Suppose $k,\theta \in \mathbb{R}$ and $h\in H_{\mathbb{C}}^{1}$. If $f(x)=e^{i(kx+\theta )}h(x)$, then
\begin{equation*}
E(f,g)=E(h,g)+k^{2}H(h)-2k\ {\rm Im} \int_{-\infty }^{\infty }h\overline{h}%
_{x}\ dx
\end{equation*}
and
\begin{equation*}
G(f,g)=G(h,g)-kH(h).
\end{equation*}
\end{lem}

\begin{lem}
\label{WrelI}
 Suppose $s>0$ and $t\in \mathbb{R}$, and define
$b=b(a)=(t-a)/s$ for $a\geq 0$. Then
\begin{equation*}
W(s,t)=\inf_{a\geq 0}\{I(s,a)+b(a)^2s\}.
\end{equation*}
\end{lem}

\begin{lem}\label{auxi}
Suppose $s>0$ and $t\in \mathbb{R}$, and define $b(a)=(t-a)/s$ for
$a\geq 0$. If $\{(h_{n},g_{n})\}$ is a minimizing sequence for
$W(s,t)$, then there exists  a subsequence (still denoted by
$\{(h_{n},g_{n})\}$) and a number $a \ge 0$ such that
\begin{equation*}
\lim_{n\to \infty}\|g_n\|^2 = a,
\end{equation*}
\begin{equation*}
\lim_{n \to \infty}E(e^{ib(a)x}h_n,g_n)=I(s,a),
\end{equation*}
and
\begin{equation}
W(s,t)=I(s,a)+b(a)^2s.
\label{wib}
\end{equation}
If $\beta_1 =0$, we can further assert that $a>0$.
\end{lem}

\begin{proof}
The sequence $a_n$ defined by
\begin{equation*}
a_n=\|g_n\|^2=G(h_n,g_n)-\text{Im}
\int_{-\infty }^{\infty }h_n\overline{h_{nx}}\ dx
=t-\text{Im}
\int_{-\infty }^{\infty }h_n\overline{h_{nx}}\ dx
\end{equation*}
is bounded, by Lemma~\ref{bounded}.  Hence, by passing to a
subsequence, we may assume that $a_n$ converges to a limit
$a\geq 0$. Let $b=b(a)$ and define $f_n(x)=e^{ibx}h_n(x)$. Then
from Lemmas \ref{decom} and \ref{WrelI} we have that
\begin{equation}
\label{lEleI}
\begin{aligned}
\lim_{n\to \infty }E(f_n,g_n)
&=\lim_{n\to\infty}\left(E(h_n,g_n)+b^2H(h_n)
- 2b\ \text{Im}\int_{-\infty}^{\infty }h_{n}\overline{h_{nx}}\ dx\right) \\
&=W(s,t)+b^2s-2b(t-a)=W(s,t)-b^2s\leq I(s,a).
\end{aligned}
\end{equation}

We claim that also
\begin{equation}
\label{lEgeI}
\lim_{n \to \infty}E(f_n,g_n) \ge I(s,a).
\end{equation}
For if $a>0$, then for sufficiently large $n$ we have that $\|f_n\|>0$ and $\|g_n\|>0$, so
the sequences $\beta_n = \sqrt{s}/\|f_n\|$ and $\theta_n = \sqrt{a}/\|g_n\|$
are defined, and both approach 1 as $n \to \infty$.  Since $\|\beta_n f_n\|^2=s$ and $\|\theta_n g_n\|^2
=a$, then $E(\beta_n f_n,\theta_n g_n)\ge I(s,a)$, and therefore
\begin{equation*}
\lim_{n\to \infty }E(f_{n},g_{n})=\lim_{n\to \infty
}E(\beta _{n}f_{n},\theta _{n}g_{n})\geq I(s,a).
\end{equation*}
On the other hand, if $a = 0$, then $\|g_n\| \to 0$ as $n \to \infty$, so it follows as in the proof
of Lemma \ref{bdgnxbelow} that \eqref{Ibbelow} holds: that is,
\begin{equation}
\label{minforszero}
\lim_{n\to \infty }E(f_n,g_n) =\lim_{n\to \infty }\int_{-\infty }^{\infty }\left( |f_{nx}|^{2}-\beta_1
|f_n|^{q+2}\right) \ dx \ge I(s,0).
\end{equation}
Hence \eqref{lEgeI} holds in either case.

All the assertions of the Lemma, except the last one, now follow from \eqref{lEleI} and \eqref{lEgeI}.

To prove the last assertion of the Lemma, assume to the contrary that $\beta_1=0$ and $a=0$.  From
Lemma \ref{bdfnxbelow} we know that $I(s,a)=0$, so from \eqref{wib} it follows that $W(s,t) \ge 0$.
But on the other hand, we can let $g_0$ be the function defined in Lemma \ref{minforJ}, and
$f_0$ be the corresponding function defined for this $g_0$ in Lemma \ref{minfforg}.   Then
$f_0$ is real, $\|f_0\|^2=s$, and $\|g_0\|^2=t$, so $H(f_0)=s$ and $G(f_0,g_0)=t$, and hence
$W(s,t) \le E(f_0,g_0)$.  Since
\begin{equation*}
E(f_0,g_0) =\int_{-\infty}^\infty\left(f_{0x}^2-\alpha f_0^2 g_0\right)\ dx + J(g_0) <0,
\end{equation*}
it follows that $W(s,t) <0$, giving the desired contradiction.
\end{proof}

We can now prove Theorem \ref{stability}.  Statement (i) of the theorem is an immediate consequence of Lemma \ref{auxi}.
To prove statement (ii), we start from a given subsequence and use Lemma \ref{auxi} to conclude that some subsequence
of $(f_n,g_n)=(e^{ibx}h_n,g_n)$ is a minimizing sequence for $I(s,a)$.

We claim that upon passing to a further subsequence, there exist real numbers $y_n$ such that
$(f_n(x+y_{n}),g_{n}(x+y_{n}))$
converges in $Y$ to some $(\phi,\psi)$ in $\mathcal S_{s,a}$.  If $a >0$, this follows immediately
from part (ii) of Theorem \ref{existence}.

If, on the other hand, $a=0$, then as in the proof of Lemma
\ref{auxi} we obtain \eqref{minforszero}.  But from \eqref{minforszero} we see that
\begin{equation*}
\lim_{n \to \infty} E(f_n,g_n)=\lim_{n \to \infty} E(f_n,0),
\end{equation*}
and since $E(f_n,g_n)$ converges to $I(s,0)$, this means that $(f_n,0)$ is a minimizing sequence for $I(s,0)$.
Since $a=0$, then Lemma \ref{auxi} implies that $\beta_1$ must be positive, so the claim follows
from Lemma \ref{Is0}.  Thus the claim has been proved in all cases.

Now, by passing to yet another subsequence, we may assume that $e^{iby_n}$ converges to
$e^{i\theta }$ for some $\theta \in [0,2\pi )$. Then $(h_{n}(.+y_{n}),g_{n}(.+y_{n}))$ converges to $(\Phi,\psi)$ in $Y$, where $\Phi(x)=e^{-i(bx+\theta )}\phi(x)$.
As  in \eqref{lEleI}, we have
\begin{equation}
\begin{aligned}
I(s,a) &=E(\phi,\psi)=E(\Phi,\psi)+b^2H(\Phi)-2b\ \ \text{Im}\int_{-\infty }^{\infty }
\Phi\overline {\Phi}_x\ dx \\
&=E(\Phi,\psi)+b^2s-2b\left(G(\Phi,\psi)-\|\psi\|^2\right) \\
&=E(\Phi,\psi)+b^2s-2b(t-s)=E(\Phi,\psi)-b^2s.
\end{aligned}
\end{equation}
It then follows from \eqref{wib} that $E(\Phi,\psi)=W(s,t)$, and hence that $(\Phi,\psi) \in
\mathcal{F}_{s,t}$.

Part (iii) of the Theorem follows from the Lagrange multiplier principle, just as did part (iii) of
Theorem \ref{existence}.

Next we prove part (iv) of Theorem \ref{stability}. Suppose
$(\Phi,\psi) \in \mathcal{F}_{s,t}$.  Applying Lemma \ref{auxi} to
the minimizing sequence for $W(s,t)$ defined by setting
$(h_n,g_n)=(\Phi,\psi)$ for all $n \in \mathbb N$, we obtain that
$(e^{ibx}\Phi,\psi)$ is a minimizing sequence for $I(s,a)$,
where $a = \|g\|^2$ and $b=(t-a)/s$. Therefore
$(e^{ibx}\Phi,\psi) \in \mathcal{S}_{s,a}$.  Hence by part (iv)
of Theorem \ref{existence}, there exist a number $\theta \in
\mathbb R$ and a real-valued function $\tilde \phi$ such that
$e^{ibx}\Phi(x)= e^{i\theta}\tilde\phi(x)$.  So
\begin{equation*}
\Phi(x)= e^{i(-bx+\theta)}\tilde\phi(x),
\end{equation*}
which is \eqref{formofphi}.  In case $\tau_1 = 0$, then $\beta_1=0$ and it follows from Lemma \ref{auxi}
that $a>0$.  Since $(\tilde \phi,\psi) \in \mathcal S_{s,a}$, it follows from part (iv) of Theorem \ref{existence}
that $\psi(x)>0$ on $\mathbb R$, and that either $\tilde \phi(x)>0$ for all $x \in \mathbb R$ or $\tilde \phi(x)<0$ for
all $x \in \mathbb R$.  In the latter case, we can add $\pi$ to the value of $\theta$ and replace $\tilde \phi$ by $e^{i\theta}\tilde \phi$
to get that $\tilde \phi$ is positive on $\mathbb R$.

Part (v) of Theorem \eqref{stability},
concerning the stability of $\mathcal F_{s,t}$, follows from
part (ii) by a standard argument, which we repeat here for
completeness.

Suppose that $\mathcal{F}_{s,t}$ is not stable. Then there exist a
number $\epsilon >0$ and sequences $(h_n,g_n)$ of initial data in
$Y$ and times $t_n\geq 0$ such that, for all $n \in \mathbb{N}$,
\begin{equation}
\label{idconvtoF} \inf\{\|(h_n,g_n)-(h,g)\|_Y : (h,g)\in
\mathcal{F}_{s,t}\}<\frac{1}{n};
\end{equation}
while the solutions $(u_n(x,t),v_n(x,t))$ of \eqref{NLSKDV} with
initial data
\begin{equation*}
(u_n(x,0),v_n(x,0))=(h_n(x),g_n(x))
\end{equation*}
satisfy
\begin{equation}\label{contra}
\inf\{\|(u_n(\cdot,t_n),v_n(\cdot,t_n)-(h,g)\|_Y : (h,g)\in
\mathcal{F}_{s,t}\} \geq \epsilon
\end{equation}
for all $n \in \mathbb N$.

 From \eqref{idconvtoF} and Lemma
\ref{Econt} we have that
\begin{equation} \label{initdata}
\begin{aligned}
\lim_{n\to \infty }E(h_n,g_n)&= W(s,t),\\
\lim_{n\to \infty }H(h_n) &=s,\\
\lim_{n\to \infty }G(h_n,g_n)&=t.
\end{aligned}
\end{equation}
Let us denote $u_n(\cdot,t_n)$ by $U_n$ and $v_n(\cdot,t_n)$ by
$V_n$. Since $E(u,v)$, $G(u,v)$, and $H(u)$ are independent of
$t$, then \eqref{initdata} implies
\begin{equation*}
\begin{aligned}
\lim_{n \to \infty }E(U_n,V_n)&=W(s,t),\\
\lim_{n \to \infty }H(U_n)&=s,\\
\lim_{n\to \infty }G(U_n,V_n)&=t,
\end{aligned}
\end{equation*}
which means that $\{(U_n,V_n)\}$, like $\{(h_n,g_n)\}$, is a
minimizing sequence for $W(s,t)$.

Now part (ii) of Theorem \ref{stability} tells us that there
exists a subsequence $\{(U_{n_k},V_{n_k})\}$, a sequence of real
numbers $\{y_k\}$, and a function pair $(\Phi,\psi) \in
\mathcal{F}_{s,t}$ such that
\begin{equation}
\lim_{k \to \infty}
\|(U_{n_k}(\cdot+y_k),V_{n_k}(\cdot+y_k))-(\Phi,\psi)\|_Y=0.
\end{equation}
So, for some sufficiently large $k$,
\begin{equation*}
\|(U_{n_k}(\cdot+y_k),V_{n_k}(\cdot+y_k))-(\Psi,\psi)\|_Y<\epsilon,
\end{equation*}
and hence
\begin{equation}
\label{penult}
\|(U_{n_k},V_{n_k})-(\Phi(\cdot-y_k),\psi(\cdot-y_k))\|_Y<\epsilon.
\end{equation}
But $(\Phi(\cdot-y_k),\psi(\cdot-y_k))$ is also in
$\mathcal{F}_{s,t}$, and hence \eqref{penult} gives
\begin{equation*}
\inf\{\|(U_{n_k},V_{n_k})-(h,g)\|_Y : (h,g)\in
\mathcal{F}_{s,t}\}<\epsilon.
\end{equation*}
Since this contradicts \eqref{contra}, we conclude that $\mathcal
F_{s,t}$ must in fact be stable.

It remains only to prove the last assertion (vi) of Theorem \ref{stability},
namely, that the sets $\mathcal F_{s,t}$ form a true two-parameter
family.  Suppose $(\Phi_1,\psi_1) \in \mathcal{F}_{s_1,t_1}$ and
$(\Phi_2,\psi_2) \in \mathcal{F}_{s_2,t_2}$, where $(s_1,t_1) \ne
(s_2,t_2)$.  We want to show $(\Phi_1,\psi_1) \ne
(\Phi_2,\psi_2)$.  If $s_1 \ne s_2$, the conclusion is obvious,
since then $\|\Phi_1\|^2 \ne \|\Phi_2\|^2$.  So we can assume $s_1
= s_2$ and $t_1 \ne t_2$. From part (iv), if we let $\eta_i =
(\|\psi_i\|^2-t_1)/s_i$ for $i=1,2$; then there exist numbers
$\theta_1$ and $\theta_2$ and real-valued functions $\tilde
\phi_1$ and $\tilde \phi_2$ such that
\begin{equation}
\Phi_1(x) = e^{i(\eta_1 x + \theta_1)}\tilde\phi_1(x) \quad\text{and}\quad
\Phi_2(x) = e^{i(\eta_2 x + \theta_2)}\tilde\phi_2(x)
\end{equation}
on $\mathbb R$.  We may assume that $\Phi_1 = \Phi_2$, or else we are done.  Then
\begin{equation*}
e^{i((\eta_1-\eta_2)x+(\theta_1-\theta_2))}=\tilde\phi_2(x)/\tilde \phi_1(x)
\end{equation*}
is real-valued on $\mathbb R$, and hence $\eta_1$ must equal $\eta_2$.  Since $s_1 =s_2$,
this implies that $\|\psi_1\|^2 -t_1= \|\psi_2\|^2-t_2$.  But $t_1 \ne t_2$, so therefore
$\|\psi_1\|^2 \ne \|\psi_2\|^2$, and hence $\psi_1 \ne \psi_2$, as desired.

The proof of Theorem \ref{stability} is now complete.

\section*{Acknowledgement}
The authors would like to thank Daniele Garrisi for an important
conversation, and in particular for bringing Lemma \ref{garlem} to
their attention.


\begin{thebibliography}{10}

\bibitem{[A]} J. Albert, \emph{Concentration compactness and the stability
of solitary-wave solutions to non-local equations}, in \emph{Applied
analysis} (ed. J. Goldstein et al.), pp. 1--29, American
Mathematical Society, Providence, 1999.

\bibitem{[AA]} J. Albert and J. Angulo Pava, \emph{Existence and stability
of ground-state solutions of a Schr\"{o}dinger-KdV system}, Proc. of the
Royal Soc. of Edinburgh, \textbf{133A} (2003), pp.\ 987--1029.

\bibitem{[An]} J. Angulo Pava, \emph{Stability of solitary wave solutions for equations of short and long dispersive waves},
Electron. J. Differential Equations, No. 72, (2006), 18 pp. (electronic).

\bibitem{[AV]} K. Appert, J. Vaclavik, \emph{Dynamics of coupled solitons}, Phys. Fluids {\bf 20} (1977), 1845--1849.

\bibitem{[BOP]} D. Bekiranov, T. Ogawa, G. Ponce, \emph{Weak solvability and well-posedness of
a coupled Schr\"{o}dinger-Korteweg-de Vries equation for capillary-gravity wave interaction},
Proc. Amer. Math. Soc., \textbf{125} (1997), 2907--2919.

\bibitem{[By]} J. Byeon, \emph{Effect of symmetry to the structure of positive solutions in nonlinear
elliptic problems}, J. Differential Equations, \textbf{163} (2000), 429-–474.

\bibitem{[Ca]} T. Cazenave, \emph{Semilinear Schrödinger equations}, Courant Lecture Notes in Mathematics, vol.\ {\textbf 10}, American Mathematical Society, Providence, 2003.

\bibitem{[C]} L. Chen, \emph{Orbital stability of solitary waves of the nonlinear Schr¨odinger-KDV equation}. J.
Partial Diff. Eqs. \textbf{12} (1999), 11–-25.

\bibitem{[CL]} A.J. Corcho and F. Linares, \emph{Well-posedness for the
Schr\"{o}dinger-Korteweg-de Vries system}, Trans. Amer. Math. Soc.,
\textbf{359} (2007), 4089--4106.

\bibitem{[Co]} J.-M. Coron, \emph{The continuity of the rearrangement in $W^{1,p}(\mathbf R)$}, Ann. Scuola Norm. Sup. Pisa Cl. Sci. $\textbf{11}\ (1984),\ 57–85.$
$\textbf{12}\ (1987),\ 1133-1173.$


\bibitem{[G]} D. Garrisi, \textit{On the orbital stability of standing-waves solutions to a
coupled non-linear Klein-Gordon equation}, arXiv : 1009.2281v2, 2011.

\bibitem{[DFO]} J. P. Dias, M. Figueira and F. Oliveira, \emph{Well-posedness and
existence of bound states for a coupled Schr\"{o}dinger-gKdV system},
Nonlinear Analysis \textbf{73} (2010), 2686--2698.

\bibitem{[F]} A. Friedman, \emph{Partial Differential Equations}, Krieger, New York, 1976.

\bibitem{[FO]} M. Funakoshi and M. Oikawa, \emph{The resonant interactions
between a long internal gravity wave and a surface gravity wave packet.} J.
Phys. Soc. Japan, \textbf{52} (1983), 1982--1995.

\bibitem{[KSK]} T. Kawahara, N. Sugimoto and T. Kakutani, \emph{Nonlinear
interaction between short and long capillary-gravity waves}, J. Phys. Soc.
Japan, \textbf{39} (1975), 1379--1386.

\bibitem{[LL]} E. H. Lieb and M. Loss, \emph{Analysis, 2nd ed.}, Graduate studies in mathematics, vol.\ \textbf{14}, American
Mathematical Society, Providence, 2001.

\bibitem{[L]} P. L. Lions, \emph{The concentration-compactness principle in
the calculus of variations. The locally compact case, Part 1 and 2,
Ann. Inst. H. Poincar\'e Analyse Non Lin\'{e}aire}, \textbf{1} (1984), 104--45 and 223--283.

\bibitem{[Lu]} D. Luenberger, \emph{Optimization by Vector Space Methods}, Wiley, New York, 1969.

\bibitem{[M]} P. Miller, \emph{Applied Asymptotic Analysis.} Graduate studies in mathematics, vol.\ \textbf{75}, American Mathematical Society, Providence, 2006.

\bibitem{[N]} A. Newell, \emph{Solitons in Mathematics and Physics.} CBMS vol. \textbf{48}, Society for Industrial and Applied Mathematics, Philadelphia, 1985.

\bibitem{[NHMI]} K. Nishikawa, H. Hojo, K. Mima and H. Ikezi. \emph{Coupled
nonlinear electron-plasma and ion-acoustic waves}, Phys. Rev. Letters \textbf{33}
(1974), 148--151.

\bibitem{[T]} T. Tao, \emph{Nonlinear Dispersive Equations: Local and Global Analysis}, CBMS vol.\ \textbf{106}, American Mathematical Society, Providence, 2006.

\bibitem{[W]} G. Whitham, \emph{Linear and Nonlinear Waves}, Wiley, 1974.

\end{thebibliography}
\end{document}